\documentclass{article}
\usepackage[english]{babel}
\usepackage{amsmath,amssymb,graphicx,latexsym,pslatex,theorem}
\usepackage{hyperref}
\usepackage{url}

\newcommand{\Alpha}{\mathrm{A}}
\newcommand{\Beta}{\mathrm{B}}
\newcommand{\assign}{:=}
\newcommand{\longdownminus}{{\mbox{\rotatebox[origin=c]{-90}{$\longminus$}}}}
\newcommand{\longminus}{{-\!\!-}}
\newcommand{\mathd}{\mathrm{d}}
\newcommand{\mathlambda}{\lambda}
\newcommand{\mathpi}{\pi}
\newcommand{\noplus}{}
\newcommand{\of}{:}
\newcommand{\tmaffiliation}[1]{\\ #1}
\newcommand{\tmmathbf}[1]{\ensuremath{\boldsymbol{#1}}}
\newcommand{\tmop}[1]{\ensuremath{\operatorname{#1}}}
\newcommand{\tmrsub}[1]{\ensuremath{_{\textrm{#1}}}}
\newcommand{\tmtextit}[1]{\text{{\itshape{#1}}}}
\newenvironment{proof}{\noindent\textbf{Proof\ }}{\hspace*{\fill}$\Box$\medskip}
\newcounter{nnacknowledgments}

{\theorembodyfont{\rmfamily}\newtheorem{acknowledgments*}[nnacknowledgments]{Acknowledgments}}
\newtheorem{corollary}{Corollary}
\newtheorem{definition}{Definition}
{\theorembodyfont{\rmfamily}\newtheorem{example}{Example}}
\newtheorem{proposition}{Proposition}
{\theorembodyfont{\rmfamily}\newtheorem{remark}{Remark}}
\newtheorem{theorem}{Theorem}
\begin{document}

\title{$q$-deformed Perelomov-Popov measures and quantized free probability}

\author{
  Panagiotis Zografos
  \tmaffiliation{Leipzig University, Mathematics Institute}
}

\maketitle

\begin{abstract}
  The asymptotic study of tuples of random non-increasing integers is crucial
  for probabilistic models coming from asymptotic representation theory and
  statistical physics. We study the global behavior of such tuples,
  introducing a new family of discrete probability measures, depending on a
  parameter $q \in [- 1, 1]$. We prove the Law of Large Numbers for these
  measures based on the asymptotics of the Schur generating functions and we
  provide explicit formulas for the moments and the free cumulants of the
  limiting measures. Our results provide an interpolation between the results of
  Bufetov and Gorin for $q = 0, 1$, who distinguished these two cases from the
  side of free probability theory. We show the connection with free
  probability theory and we introduce a deformation of free convolution,
  motivated by our formulas for the free cumulants. We also study the first
  order correction to the Law of Large Numbers and we make the connection with
  infinitesimal free probability, computing explicitly the infinitesimal
  moments and the infinitesimal free cumulants. Finally, we prove
  non-asymptotic relations between the limiting measures for different $q$,
  which are related to the celebrated Markov-Krein correspondence.
\end{abstract}

\section{Introduction}

\subsection{Overview}

In the current article, we introduce a new family of probability measures that
encode random signatures and study its asymptotics. Let $\hat{U} (N)$ be the
set of all signatures $\mathlambda (N)$, i.e. $N$-tuples $\mathlambda (N) =
(\mathlambda_1 (N), \ldots, \mathlambda_N (N)) \in \mathbb{Z}^N$ of
non-increasing integers. One of key conections that motivates the study of
such simple objects is that they are in one-to-one correspondence with the
isomorphism classes of the irreducible representations of the unitary group $U
(N)$ of all $N \times N$ complex unitary matrices (see e.g. {\cite{B37}}).

We study the global asymptotics of $\mathlambda (N)$, as $N \rightarrow
\infty$, proving among other the Law of Large Numbers for discrete probability
measures, parametrized by random $\mathlambda (N)$. For this purpose we
introduce the probability measures
\begin{equation}
  m_{N, P \text{} P (q)} [\mathlambda (N)] \assign \frac{1}{N} \sum_{i = 1}^N
  \left( \prod_{j \neq i} \frac{(\mathlambda_i (N) - i) - (\mathlambda_j (N) -
  j) - q}{(\mathlambda_i (N) - i) - (\mathlambda_j (N) - j)} \right) \delta
  \left( \frac{\mathlambda_i (N) + N - i}{N} \right)
  \label{PerelomovPopovmeasure},
\end{equation}
where $q \in [- 1, 1]$. Although it is clear that for $q = 0$ this is a very
natural measure, up to the author's knowledge, this measure has not been
considered in the literature before for $q \neq 0, 1$.

Our theorems are based on general assumptions for probability measures on
$\hat{U} (N)$, related to a representation-theoretic notion of Fourier
transform for such measures. We recall that a character
$\tmmathbf{\chi}^{\pi_N}  : U (N) \rightarrow \mathbb{C}$ of a
finite-dimensional representation $\pi_N$ of $U (N)$ is defined by
$\tmmathbf{\chi}^{\pi_N} (U) = \tmop{Tr} (\pi_N (U))$, for every $U \in U
(N)$. Since $U \in U (N)$ is diagonalizable $\tmmathbf{\chi}^{\pi_N} (U)$
depends only on the spectrum of $U$ and $\tmmathbf{\chi}^{\pi_N}$ is a
function of $N$ variables on the unit circle. For $\mathlambda (N) \in \hat{U}
(N)$ we denote by $\tmmathbf{\chi}^{\mathlambda (N)}$ the character of the
irreducible representation $\pi^{\mathlambda (N)}$ that corresponds to the
signature $\mathlambda (N)$. These characters are the rational Schur
functions,
\[ \tmmathbf{\chi}^{\mathlambda (N)} (u_1, \ldots, u_N) = \frac{\det
   (u_i^{\mathlambda_j (N) + N - j})_{1 \leq i, j \leq N}}{\det (u_i^{N -
   j})_{1 \leq i, j \leq N}} = \frac{\det (u_i^{\mathlambda_j (N) + N - j})_{1
   \leq i, j \leq N}}{\prod_{1 \leq i < j \leq N} (u_i - u_j)} . \]
Let $\varrho (N)$ be a probability measure on $\hat{U} (N)$. The Schur
generating function $S_{\varrho (N)}$, with respect to $\varrho (N)$, is
defined by
\[ S_{\varrho (N)} (u_1, \ldots, u_N) \assign \mathbb{E}_{\mathlambda (N)}
   \left[ \frac{\tmmathbf{\chi}^{\mathlambda (N)} (u_1, \ldots,
   u_N)}{\tmmathbf{\chi}^{\mathlambda (N)} (1, \ldots, 1)} \right] =
   \sum_{\mathlambda (N) \in \hat{U} (N)} \varrho (N) [\mathlambda (N)] 
   \frac{\tmmathbf{\chi}^{\mathlambda (N)} (u_1, \ldots,
   u_N)}{\tmmathbf{\chi}^{\mathlambda (N)} (1, \ldots, 1)} . \]
In the following, we assume that the Schur generating functions converge
uniformly in a neighborhood of $1^N \assign (1, \ldots, 1)$. This can be
guaranteed under very mild assumptions for $\varrho (N)$ (see e.g.
{\cite{B4}}). One can think of $S_{\varrho (N)}$ as a Fourier transform on $U
(N)$ because it is analogous to the definition of the Fourier transform of a
probability measure on $\mathbb{R}$, in the sense that they are expectations
of random characters.

We denote by $m_{N, P \text{} P (q)} [\varrho (N)]$ the random probability
measure (\ref{PerelomovPopovmeasure}), where $\mathlambda (N)$ is distributed
according to $\varrho (N)$. Our first main result is the Law of Large Numbers
(LLN) for $m_{N, P \text{} P (q)} [\varrho (N)] .$

\begin{theorem}
  \label{MAINTHEOREM}Let $\varrho (N)$, $N \in \mathbb{N}$, be a sequence of
  probability measures on $\hat{U} (N)$ such that
  \[ \lim_{N \rightarrow \infty} \partial_{i_0}^k \left( \frac{1}{N} \log
     \text{} S_{\varrho (N)} \right) (1^N) = a_k \text{\quad and\quad} \lim_{N
     \rightarrow \infty} \partial_{i_1} \ldots \partial_{i_m} \left(
     \frac{1}{N} \log \text{} S_{\varrho (N)} \right) (1^N) = 0, \]
  for every $k \in \mathbb{N}$ and $i_0, \ldots, i_m \in \{1, \ldots, N\}$
  with $|\{i_1, \ldots, i_m \}| \geq 2$. Moreover assume that the power series
  \[ \Psi (z) \assign \sum_{k = 0}^{\infty} a_k  \frac{(z - 1)^k}{k!}, \]
  converges in a neighborhood of $1$. Then the sequence of random measures
  $m_{N, P \text{} P (q)} [\varrho (N)]$ converges as $N \rightarrow \infty$
  in probability, in the sense of moments to a deterministic probability
  measure $\tmmathbf{\mu}^{(q)}$ on $\mathbb{R}$ whith moments
  \begin{equation}
    \tmmathbf{\mu}_k^{(q)} = \sum_{m = 0}^k \binom{k}{m} \frac{1}{(m + 1) !} 
    \left. \frac{\mathd^m}{\mathd u^m} (u^{k - q} (\Psi' (u))^{k - m})
    \right|_{u = 1} . \label{ropes}
  \end{equation}
  For $q, q' \in [- 1, 1]$ the limiting measures of $m_{N, P \text{} P (q)}
  [\varrho (N)]$, $m_{N, P \text{} P (q')} [\varrho (N)]$ are related through
  \begin{equation}
    \left( 1 - q \sum_{k = 0}^{\infty} \tmmathbf{\mu}_k^{(q)} z^{k + 1}
    \right)^{1 / q} = \left( 1 - q' \sum_{k = 0}^{\infty}
    \tmmathbf{\mu}_k^{(q')} z^{k + 1} \right)^{1 / q'} .
    \label{TRANSFORMATION}
  \end{equation}
\end{theorem}

The representation-theoretic interpretation of Schur generating functions
leads to applications of Theorem \ref{MAINTHEOREM} in asymptotic
representation theory of the unitary group. More precisely, we provide new
results for the decomposition into irreducible components of tensor products
of irreducible representations of large unitary groups. For $q = 0, 1$,
similar or more general results have already been shown in this direction
{\cite{B20}}, {\cite{B1}}, {\cite{B3}}, {\cite{B4}}. For a finite-dimensional
representation $\pi_N$ of $U (N)$ we have a decomposition into irreducible
components:
\[ \pi_N = \bigoplus_{\mathlambda (N) \in \hat{U} (N)} c_{\mathlambda (N)}
   \pi^{\mathlambda (N)}, \]
which gives rise to the sequence of probability measures on $\hat{U} (N)$
\[ \varrho_{\pi_N} [\mathlambda (N)] = \frac{c_{\mathlambda (N)} \dim
   (\pi^{\mathlambda (N)})}{\dim (\pi_N)} \text{, \quad for every }
   \mathlambda (N) \in \hat{U} (N) . \]
In the above $c_{\mathlambda (N)}$ are the multiplicities and dim stands for
the dimension of the representation. For $\mathlambda_1 (N), \mathlambda_2 (N)
\in \hat{U} (N)$ and $\pi_N = \pi^{\mathlambda_1 (N)} \otimes
\pi^{\mathlambda_2 (N)}$ the fact that $S_{\varrho_{\pi_N}} =
S_{\varrho_{\pi^{\mathlambda_1 (N)}}} \cdot S_{\varrho_{\pi^{\mathlambda_2
(N)}}}$ leads to the following corollary.

\begin{corollary}
  \label{COROLLARYMAIN}Let $\mathlambda_1 (N), \mathlambda_2 (N) \in \hat{U}
  (N)$, $N \in \mathbb{N}$, be two sequences of signatures that satisfy some
  technical assumptions, including that
  \[ \lim_{N \rightarrow \infty} m_{N, P \text{} P (q)} [\mathlambda_i (N)]
     =\tmmathbf{\mu}_i \text{\quad for } i = 1, 2, \]
  where the above convergence is in the sense of moments. Then, if $\pi_N =
  \pi^{\mathlambda_1 (N)} \otimes \pi^{\mathlambda_2 (N)}$, the sequence of
  random measures $m_{N, P \text{} P (q)} [\varrho_{\pi_N}]$ converges as $N
  \rightarrow \infty$ in probability, in the sense of moments to a
  deterministic probability measure on $\mathbb{R}$ which we denote by
  $\tmmathbf{\mu}_1 \otimes_q \tmmathbf{\mu}_2$, and it is such that
  \[ (\tmmathbf{\mu}_1 \otimes_q \tmmathbf{\mu}_2) \boxplus \beta (1 - q, 1 +
     q) =\tmmathbf{\mu}_1 \boxplus \tmmathbf{\mu}_2, \]
  where $\beta (1 - q, 1 + q)$ is the beta distribution with parameters $1 -
  q, 1 + q \geq 0$ and $\boxplus$ stands for the free convolution.
\end{corollary}

The technical assumptions mentioned in the above corollary will not be
important for us due to the fact that the convergence for $m_{N, P \text{} P (q)}
[\mathlambda (N)]$ is equivalent to convergence for $m_{N, P \text{} P (0)}
[\mathlambda (N)]$. For more details we refer to {\cite{B1}}.

Staring with three sequences $\varrho_1 (N), \varrho_2 (N), \varrho_3 (N)$ of
probability measures on $\hat{U} (N)$ such that they satisfy the assumptions of Theorem \ref{MAINTHEOREM} and
$S_{\varrho_3 (N)} = S_{\varrho_1 (N)} \cdot S_{\varrho_2 (N)}$ one can ask if
there is a meaningful expression for the limit of $m_{N, P \text{} P (q)}
[\varrho_3 (N)]$ in terms of the limits of $m_{N, P \text{} P (q)} [\varrho_1
(N)]$ and $m_{N, P \text{} P (q)} [\varrho_2 (N)]$. To answer this question we
make the connection with free probability and we introduce a $q$-deformation
of free convolution (see Theorem \ref{kolpa}).

The techniques used in the current article lead to the asymptotic expansion for
the measure $\mathbb{E} [m_{N, P \text{} P (q)} [\varrho (N)]]$. We
also study the second leading term of this expansion. We will often refer to
it as the correction (to the Law of Large Numbers) or as the infinitesimal limit. To
make it precise, given a sequence of probability measures $\varrho (N)$ on
$\hat{U} (N)$, we define the sequence of linear functionals
\[ \varphi_N^{(q)} (P) \assign \mathbb{E}_{\mathlambda (N)} \left[
   \int_{\mathbb{R}} P (t) m_{N, P \text{} P (q)} [\mathlambda (N)] (d \text{}
   t) \right] \text{, \quad} P \in \mathbb{C} [\tmmathbf{t}] . \]
By the infinitesimal limit (of $\{\varphi_N^{(q)} \}_{N \in \mathbb{N}}$), we mean
the triple $(\mathbb{C}[\tmmathbf{t}], \varphi_q, \varphi_q')$, where
\[ \varphi_q (P) = \lim_{N \rightarrow \infty} \varphi_N^{(q)} (P) \text{\quad
   and\quad} \varphi_q' (P) = \lim_{N \rightarrow \infty} N (\varphi_N^{(q)}
   (P) - \varphi_q (P)) \text{, \quad for } P \in \mathbb{C} [\tmmathbf{t}] .
\]
\begin{theorem}
  \label{infinitesimaltransform}Let $\varrho (N)$, $N \in \mathbb{N}$, be a
  sequence of probability measures on $\hat{U} (N)$ such that
  \[ \lim_{N \rightarrow \infty} \left( \partial_{i_0}^k \log \text{}
     S_{\varrho (N)} (1^N) - N \text{} a_k \right) = b_k \text{\quad and\quad}
     \lim_{N \rightarrow \infty} \partial_{i_1} \ldots \partial_{i_m} \log
     \text{} S_{\varrho (N)} (1^N) = 0, \]
  for every $k \in \mathbb{N}$ and $i_0, \ldots, i_m \in \{1, \ldots, N\}$
  with $|\{i_1, \ldots, i_m \}| \geq 2$. Moreover, assume that $(a_k)_{k \in
  \mathbb{N}}, (b_k)_{k \in \mathbb{N}}$ are such that
  \[ \Psi (z) = \sum_{k = 0}^{\infty} a_k  \frac{(z - 1)^k}{k!} \text{\quad
     and\quad} \Phi (z) = \sum_{k = 0}^{\infty} b_k  \frac{(z - 1)^k}{k!}, \]
  converge in a neighborhood of $1$. Then, for every $q \in [- 1, 1]$, there
  exists an infinitesimal limit $(\mathbb{C}[\tmmathbf{t}], \varphi_q,
  \varphi'_q)$ of $\{ \varphi_N^{(q)} \}_{N \in \mathbb{N}}$. The moments
  $\varphi_q (\tmmathbf{t}^k)$ are given by (\ref{ropes}). We also have the
  explicit formulas
  \[ \varphi_q' (\tmmathbf{t}^k) = \sum_{m = 0}^{k - 1} \binom{k}{m + 1}
     \frac{1}{m!}  \left. \frac{\mathd^m}{\mathd u^m} \left( u^{k - q} \left(
     \Phi' (u) - \frac{1 - q}{2 u} \right) (\Psi' (u))^{k - m - 1} \right)
     \right|_{u = 1} . \]
  Moreover, for $\tmmathbf{\mu}_k^{(q)} = \varphi_q (\tmmathbf{t}^k)$,
  $\tmmathbf{\mu}_k^{(0)} = \varphi_0 (\tmmathbf{t}^k)$,
  ${\tmmathbf{\mu}_k^{(q)}}' = \varphi_q' (\tmmathbf{t}^k)$,
  ${\tmmathbf{\mu}_k^{(0)}}' = \varphi_0' (\tmmathbf{t}^k)$, $k \in
  \mathbb{N}$, we have that
  \begin{equation}
    \exp \left( q \sum_{k = 0}^{\infty} \tmmathbf{\mu}_k^{(0)} z^{k + 1}
    \right) \sum_{k = 0}^{\infty} {\tmmathbf{\mu}_k^{(q)}}' z^{k + 1} =
    \sum_{k = 0}^{\infty} {\tmmathbf{\mu}_k^{(0)}}' z^{k + 1} + \frac{q}{2}
    \sum_{k = 0}^{\infty} (k + 1) \tmmathbf{\mu}_k^{(0)} z^{k + 2} .
    \label{infinitesimalSt}
  \end{equation}
\end{theorem}

Infinitesimal limits of averaged empirical measures of random matrices have
been studied extensively the last few years, in the context of infinitesimal
free probability. We also make the connection with infinitesimal free
probability and we introduce a $q$-deformation of the infinitesimal free
convolution (see Theorem \ref{MN}).

To summarize, in this article on the one hand we prove asymptotic results for
$m_{N, P \text{} P (q)} [\varrho (N)]$, based on
asymptotic conditions for the Schur generating functions $S_{\varrho (N)}$ and
on the other hand we show how the limiting measures that correspond to
different $q \in [- 1, 1]$ are related in a non-asymptotic manner, through
transforms of probability measures on $\mathbb{R}$. We start by explaining a
combinatorial interpretation of the measures $m_{N, P \text{} P (q)}
[\mathlambda (N)]$ showing how they are related to standard symmetric
polynomials (see Proposition \ref{Prop9}). Their combinatorial side is crucial
for the non-asymptotic relations of the limits, which are corollaries of
Newton identities.

We show two different ways to prove the LLN for $m_{N, P \text{} P (q)}
[\varrho (N)]$ (see Theorem \ref{12} and Theorem \ref{kolpa}). In both, we use
the same technique which is based on applying differential operators to
analytic and symmetric functions. Although both of them give a very precise
description of the limiting measures through their moments, the second one
allows us to go a step further and determine precisely their free cumulants,
which cannot be extracted in a straightforward way from the formulas for the
moments that we get from the first approach. The explicit formulas for the
free cumulants describe the impact of the conditions of Theorem
\ref{MAINTHEOREM} for the Schur generating functions, on the limiting
measures. They also reveal interesting properties for these measures that lead
to the results of Corollary \ref{COROLLARYMAIN} and motivate an operation
$\otimes_q$ for probability measures on $\mathbb{R}$ with compact support,
that are approximated by $m_{N, P \text{} P (q)} [\mathlambda (N)]$,
$\mathlambda (N) \in \hat{U} (N)$. This operation is strongly related to the
free convolution.

We also investigate asymptotic properties of $m_{N, P \text{} P (q)} [\varrho
(N)]$ under different limit regimes for $S_{\varrho (N)}$. Though these
different limit regimes do not change the limiting measures
$\tmmathbf{\mu}^{(q)}$, they affect interestingly the $1 / N$ correction of
$\tmmathbf{\mu}^{(q)}$, meaning the $\lim_{N \rightarrow \infty} N \left(
\mathbb{E}[m_{N, P \text{} P (q)} [\varrho (N)]] -\tmmathbf{\mu}^{(q)}
\right)$. Similarly to the above, we study the $1 / N$ correction via its
moments using both methods of differential operators mentioned above (see
Theorem \ref{meis} and Theorem \ref{MN}). The key tool for our analysis
is that of infinitesimal free cumulants. We show how the additional
requirements of the second limit regime give rise to a $1 / N$ correction,
determined by very concrete free cumulants/infinitesimal free cumulants. This
leads to results similar to that of Corollary \ref{COROLLARYMAIN} but on the
more general infinitesimal level where the infinitesimal free convolution is
involved. It also motivates a deformation of the infinitesimal free
convolution.

We make the connection of our non-asymptotic relations for the limiting
measures to the celebrated Markov-Krein correspondence (see Theorem
\ref{Xri}). We also point how these non-asymptotic relations emerge from
standard tools of free probability that were used for the understanding of
these measures.

Finally, we also present an example, where using Theorem \ref{9} we compute
explicitly the probability density function of the limit of $m_{N, P \text{} P
(q)} [\varrho (N)]$ (see Example \ref{mainexample}). Interestingly, we get a
probability measure that gives an interpolation between the semicircle
distribution, the Marchenko-Pastur distribution and the one-sided Plancherel
distribution. As far as we know this probability distribution has not appeared
in the literature before. We think that it would be interesting to understand
if it is connected to some random matrix ensemble which gives an interpolation
between the Gaussian unitary ensemble and Wishart matrices.

\subsection{Free convolution and its quantized versions}

We briefly recall the notions of free cumulants and free convolution and we
briefly describe the kind of deformations of free convolution that we consider
in this article. The free convolution is a well studied operation of
probability measures on $\mathbb{R}$ with compact support. To justify the
name, it is a convolution for non-commutative random variables (such as random
matrices or operators on Hilbert spaces) that have a non-commutative notion of
stochastic independence known as freeness (see e.g. {\cite{B15}},
{\cite{B22}}). It plays an essential role to the study of random matrices as
their size goes to infinity.

Compared to the classical convolution, free convolution is a much more
complicated operation of probability measures. Although the $k$-th moment of
the classical convolution $\tmmathbf{\mu}_1 \ast \tmmathbf{\mu}_2$ of two
probability measures $\tmmathbf{\mu}_1, \tmmathbf{\mu}_2$ can be expressed
very simply in terms of the moments of $\tmmathbf{\mu}_1$ and
$\tmmathbf{\mu}_2$, this is not the case for the $k$-th moment of
$\tmmathbf{\mu}_1 \boxplus \tmmathbf{\mu}_2$. In that sense given three
probability measures on $\mathbb{R}$ with compact support, looking their
moments it is hard to deduce whether one of them if the free convolution of
the other two. This problem can be solved by acquiring a free analogue of the
cumulants.

\begin{definition}[Speicher, {\cite{B15}}]
  Let $\tmmathbf{\mu}$ be a probability measure on $\mathbb{R}$ with compact
  support. The free cumulants $(\kappa_n (\tmmathbf{\mu}))_{n \in \mathbb{N}}$
  of $\tmmathbf{\mu}$ are uniquely determined by
  \begin{equation}
    \int_{\mathbb{R}} t^k \tmmathbf{\mu} (d \text{} t) = \sum_{\pi \in
    \tmop{NC} (k)} \prod_{V \in \pi} \kappa_{|V|} (\tmmathbf{\mu}) \text{,
    \quad for every } k \in \mathbb{N} \label{momentcumulantrelations},
  \end{equation}
  where $\tmop{NC} (k)$ stands for the lattice of non-crossing partitions of
  $\{1, \ldots, k\}$.
\end{definition}

We will refer to (\ref{momentcumulantrelations}) as moment-cumulant relation.
Similarly to the classical case, free cumulants linearize the free convolution
(see e.g. {\cite{B15}}), meaning that
\[ \kappa_n (\tmmathbf{\mu}_1 \boxplus \tmmathbf{\mu}_2) = \kappa_n
   (\tmmathbf{\mu}_1) + \kappa_n (\tmmathbf{\mu}_2) \text{, \quad for every }
   n \in \mathbb{N}, \]
Similarly, as we will prove later on, being $\tmmathbf{\mu}_1,
\tmmathbf{\mu}_2$ limits of measures $m_{N, P \text{} P (q)} [\mathlambda
(N)]$, the measure $\tmmathbf{\mu}_1 \otimes_q \tmmathbf{\mu}_2$ of Corollary
\ref{COROLLARYMAIN} is determined by the linearization of another analytic
function. This function will play the role of $R$-transform for $\otimes_q$
and it is given by
\[ R_{\tmmathbf{\mu}}^{q \text{} u \text{} a \text{} n \text{} t (q)} (z) =
   R_{\tmmathbf{\mu}} (z) - R_{\beta (1 - q, 1 + q)} (z) = R_{\tmmathbf{\mu}}
   (z) - \frac{(1 - q \text{} z)^{- 1 / q}}{(1 - q \text{} z)^{- 1 / q} - 1} +
   \frac{1}{z} . \]
We refer to $R_{\tmmathbf{\mu}}^{q \text{} u \text{} a \text{} n \text{} t
(q)}$ as $q$-deformed quantized $R$-transform. Note that
$R_{\tmmathbf{\mu}}^{q \text{} u \text{} a \text{} n \text{} t (1)} =
R_{\tmmathbf{\mu}}$ and $R_{\tmmathbf{\mu}}^{q \text{} u \text{} a \text{} n
\text{} t (- 1)} = R_{\tmmathbf{\mu} \boxplus \delta (- 1)}$.

\subsection{Related works}

The technique of Schur generating functions is not a new one. It has been
developed in the last ten years for the study of the asymptotic behavior of
discrete particle systems and random matrices {\cite{B1}}, {\cite{B3}},
{\cite{B4}}, {\cite{B7}}. Similar approaches exist for generating functions of
Jack polynomials {\cite{B11}}. Related methods for local asymptotics have been
presented in {\cite{B20}}. The LLN for $m_{N, P \text{} P (q)} [\mathlambda
(N)]$ for the particular cases $q = 0, 1$ was firstly proved in {\cite{B1}}
where they also highlighted the connection to free probability. As we already
mentioned, for $q = 1$ one gets the classical $R$-transform but for $q = 0$
the quantized $R$-transform becomes
\[ R_{\tmmathbf{\mu}}^{q \text{} u \text{} a \text{} n \text{} t (0)} (z) =
   R_{\tmmathbf{\mu}} (z) - R_{u [0, 1]} (z) = R_{\tmmathbf{\mu}} (z) -
   \frac{e^z}{e^z - 1} + \frac{1}{z}, \]
where $u [0, 1]$ denotes the uniform measure on $[0, 1]$. In {\cite{B1}} they
proved that to get an operation of probability measures that linearizes
$\tmmathbf{\mu} \mapsto R_{\tmmathbf{\mu}}^{q \text{} u \text{} a \text{} n
\text{} t (0)}$ one has to restrict to compactly supported probability
measures on $\mathbb{R}$, continuous with respect to the Lebesque measure and
such that their Radon-Nikodym derivative is bounded by one.

We introduce the measures $m_{N, P \text{} P (q)} [\mathlambda (N)]$, for $q
\in [- 1, 1]$, to provide an interpolation to the already known results for $q
= 0, 1$. As mentioned, we are not aware of other works defining such measures
and studying their asymptotic properties. Recently, it was proved in
{\cite{B40}} that the moments of $m_{N, P \text{} P (q)} [\mathlambda (N)]$
are eigenvalues of Heckman-Polychronakos operators for symmetric Jack
polynomials.

In {\cite{B1}}, $m_{N, P \text{} P (1)} [\mathlambda (N)]$ is referred to as
Perelomov-Popov measure because it is related to the work of Perelomov and
Popov {\cite{B26}} on the centers of universal enveloping algebras of
classical Lie groups. To keep it similar we will refer to $m_{N, P \text{} P
(q)} [\mathlambda (N)]$ as $q$-deformed Perelomov-Popov measure. Therefore, on
the one hand, for $q = 1$ the Perelomov-Popov measures have a clear algebraic
interpretation, while on the other hand, for $q = 0$, they have been used in the
literature to study geometric properties of tilings {\cite{B1}}, {\cite{B3}},
{\cite{B4}}, {\cite{B45}}, {\cite{B41}}. Or results provide a natural continuous interpolation between these two quite distinct setups.

Relevant approaches with symmetric generating functions exist to study the
spectrum of self-adjoint random matrices as their size goes to infinity
{\cite{B6}}, {\cite{B41}}, {\cite{B25}}, {\cite{B8}}, {\cite{B10}},
{\cite{B33}}. These are mainly based on the asymptotics of the Harish-Chandra
transforms. The framework of random matrices can be seen as the continuous
analogue of ours. It is not clear to the author if some meaningful
$q$-generalization of the empirical spectral distribution could be considered
to play the role of $m_{N, P \text{} P (q)} [\mathlambda (N)]$ for $q \neq 0$.
The infinitesimal limit of $m_{N, P \text{} P (0)} [\mathlambda (N)]$
presented on this article was also studied in {\cite{B41}}, in order to
understand the connection of random signatures to infinitesimal free
probability in an analogous way that the connection to free probability was
understood in {\cite{B1}}, and make the comparison with the results for random
matrices.

Moreover, there exist other transforms, relevant to (\ref{TRANSFORMATION}) in
the literature between measures or distributions, that emerge trying to relate
the limits of random (signed) measures that encode signatures or eigenvalues
{\cite{B6}}, {\cite{B5}}, {\cite{B1}}, {\cite{B32}}, {\cite{B21}},
{\cite{B33}}. It ends up that many of them are related to celebrated
transforms that originate from problems of functional analysis {\cite{B35}},
{\cite{B19}}, {\cite{B18}}, {\cite{B23}}.

\begin{acknowledgments*}
  The author is grateful to Alexey Bufetov for many valuable discussions and
  the encouragement to work on this project. The author is also grateful to
  Vadim Gorin for helpful discussions and Grigori Olshanski for valuable
  comments. The author was partially supported by the European Research
  Council (ERC), Grant Agreement No. 101041499.
\end{acknowledgments*}

\section{Law of Large Numbers for Perelomov-Popov measures}\label{.}

In the current section we prove the Law of Large Numbers for $m_{N, P \text{}
P (q)} [\varrho (N)]$. Before we formulate and prove our results, we start by
introducing some differential operators that will be important for our
analysis. Applying these differential operators to the Schur generating
functions we can compute the moments of $m_{N, P \text{} P (q)} [\mathlambda
(N)]$ in a similar manner that one can compute moments of a random variable
differentiating its characteristic function.

\subsection{Schur generating functions and differential operators}

We start by proving that $m_{N, P \text{} P (q)} [\mathlambda (N)]$ is indeed
a probability measure.

\begin{proposition}
  Let $\mathlambda (N) = (\mathlambda_1 (N) \geq \cdots \geq \mathlambda_N
  (N)) \in \mathbb{Z}^N$ be a signature and $q \in \mathbb{R}$. Then $m_{N, P
  \text{} P (q)} [\mathlambda (N)]$ is a signed measure of total mass $1$ and
  a probability measure for $q \in [- 1, 1]$.
\end{proposition}

\begin{proof}
  Since $\{\mathlambda_i (N) + N - i\}_{i = 1}^N$ are strictly decreasing
  integers, $m_{N, P \text{} P (q)} [\mathlambda (N)]$ is a positive measure
  for $q \in [- 1, 1]$. For notation simplicity, consider $x_i (N) =
  \mathlambda_i (N) - i$ for every $i = 1, \ldots, N$. Then for $i_0 \in \{1,
  \ldots, N\}$, the product that corresponds to $i_0$
  \[ \prod_{j \neq i_0} \frac{x_{i_0} (N) - x_j (N) - q}{x_{i_0} (N) - x_j
     (N)} = \prod_{j \neq i_0} \left( 1 + \frac{- q}{x_{i_0 (N)} - x_j (N)}
     \right), \]
  is equal to a big sum where its summands are either $1$ or
  \[ a_{\{i_1, \ldots, i_l \}} (i_0) \assign \frac{(- q)^l}{(x_{i_0} (N) -
     x_{i_1} (N)) \ldots (x_{i_0} (N) - x_{i_l} (N))} \text{\quad for some } l
     = 1, \ldots, N - 1, \]
  where $i_0, \ldots, i_l \in \{1, \ldots, N\}$ are distinct. But these summands will not contribute to the mass of the measure since
  \[ a_{\{i_1, \ldots, i_l \}} (i_0) + a_{\{i_0, i_2, \ldots, i_l \}} (i_1) +
     \cdots + a_{\{i_0, \ldots, i_{l - 1} \}} (i_l) = 0. \]
  This implies that the mass is equal to $1$ for every $q \in \mathbb{R}$.
\end{proof}

Although the $q$-deformed Perelomov-Popov measures do not seem very
``natural'' from the probabilistic point of view, there are some interesting
properties from the side of symmetric functions. These can be understood
trying to express the moments of $m_{N, P \text{} P (q)} [\mathlambda (N)]$ in
terms of power sums on the variables $x_i (N) = \mathlambda_i (N) + N - i$, $i
= 1, \ldots, N$.

\begin{proposition}
  \label{Prop9}Let $x_1, \ldots, x_N$ be $N$ distinct variables and for every
  $k \in \mathbb{N}$ consider the rational function
  \begin{equation}
    m_k^{(q)} (x_1, \ldots, x_N) \assign \sum_{i = 1}^N \left( \prod_{j \neq
    i} \frac{x_i - x_j - q}{x_i - x_j} \right) x_i^k . \label{fexi}
  \end{equation}
  Then $m_k^{(q)} (x_1, \ldots, x_n)$ is a symmetric polynomial that satisfies
  the relation
  \[ (k + 1) !m_k^{(q)} (x_1, \ldots, x_N) \]
  \begin{equation}
    \text{\qquad} = \frac{1}{q} \sum_{\pi \in P (k + 1)} (- 1)^{| \pi | - 1}
    \prod_{V \in \pi} (|V| - 1) ! (p_{|V|} (x_1 + q, \ldots, x_N + q) -
    p_{|V|} (x_1, \ldots, x_N)), \label{7}
  \end{equation}
  where $p_k (x_1, \ldots, x_N) \assign x_1^k + \cdots + x_N^k$ and $P (k)$
  denotes the set of partitions of $\{1, \ldots, k\}$.
\end{proposition}

\begin{proof}
  By relation (\ref{fexi}) we have that
  \[ m_k^{(q)} (x_1, \ldots, x_N) = \sum_{l = 1}^N \frac{(- q)^{l - 1}}{l!}
     \sum_{\underset{i (\alpha) \neq i (\beta) \text{ for } \alpha \neq
     \beta}{i (1), \ldots, i (l) = 1}}^N h_{k - l + 1} (x_{i (1)}, \ldots,
     x_{i (l)}), \]
  where $h_k$ denotes the complete homogeneous symmetric polynomial of degree
  $k$. This implies that
  \begin{eqnarray*}
    \sum_{k = 0}^{\infty} m_k^{(q)} (x_1, \ldots, x_N) z^{k + 1} & = & \sum_{l
    = 1}^N \frac{(- q)^{l - 1} z^l}{l!} \sum_{\underset{i (\alpha) \neq i
    (\beta) \text{ for } \alpha \neq \beta}{i (1), \ldots, i (l) = 1}}^N
    \sum_{k = l - 1}^{\infty} h_{k - l + 1} (x_{i (1)}, \ldots, x_{i (l)})
    z^{k - l + 1}\\
    & = & \sum_{l = 1}^N \frac{(- q)^{l - 1} z^l}{l!} \sum_{\underset{i
    (\alpha) \neq i (\beta) \text{ for } \alpha \neq \beta}{i (1), \ldots, i
    (l) = 1}}^N \prod_{j = 1}^l \frac{1}{1 - x_{i (j)} z}\\
    & = & - \sum_{l = 1}^N \frac{q^{l - 1} z^l}{l!}  \frac{(\partial_{x_1} +
    \cdots + \partial_{x_N})^l \prod_{i = 1}^N (1 - x_i z)}{\prod_{i = 1}^N (1
    - x_i z)}\\
    & = & - \left( \prod_{i = 1}^N \frac{1}{1 - x_i z} \right)  \sum_{l =
    1}^N \frac{q^{l - 1}}{l!}  \left. \frac{\mathd^l}{\mathd t^l} f (z, x_1 +
    t, \ldots, x_N + t) \right|_{t = 0},
  \end{eqnarray*}
  where $f (z, x_1, \ldots, x_N) \assign \prod_{i = 1}^N (1 - x_i z)$. Thus $f
  (z, x_1 + t, \ldots, x_N + t)$ is a polynomial of $t$ of degree $N$ and by
  Taylor's theorem we have that
  \begin{eqnarray*}
    \sum_{k = 0}^{\infty} m_k^{(q)} (x_1, \ldots, x_N) z^{k + 1} & = &
    \frac{1}{q} - \frac{1}{q}  \prod_{i = 1}^N \frac{1 - (x_i + q) z}{1 - x_i
    z}\\
    & = & \frac{1}{q} - \frac{1}{q} \exp \left( \sum_{k = 1}^{\infty}
    \frac{p_k (x_1, \ldots, x_N) - p_k (x_1 + q, \ldots, x_N + q)}{k} z^k
    \right) .
  \end{eqnarray*}
  Differentiating the above relation $k + 1$ times and setting $z = 0$, we get
  relation (\ref{7}).
\end{proof}

\begin{remark}
  It seems that the polynomials $m_k^{(q)} (x_1, \ldots, x_N)$ are very
  natural functions from the side of symmetric polynomials. In the theory of
  symmetric functions (see e.g. {\cite{B34}}) the super-symmetric complete
  homogeneous polynomials $h_k (x_1, \ldots, x_N ; y_1 {, \ldots, y_N} )$, $k
  \in \mathbb{N}$, are symmetric in $(x_1, \ldots, x_N)$, $(y_1, \ldots, y_N)$
  separately and such that the super-symmetric power sums
  \[ p_k (x_1, \ldots, x_N ; y_1, \ldots, y_N) = \sum_{i = 1}^N x_i^k -
     \sum_{i = 1}^N (- y_i)^k, \]
  can be expressed in terms of these in the usual way
  \[ \sum_{k = 0}^{\infty} h_k z^k = \exp \left( \sum_{k = 1}^{\infty}
     \frac{p_k}{k} z^k \right) . \]
  The formula that we proved in the above proposition for the generating
  function of the symmetric polynomials
  \[ h_k^{(+ q)} (x_1, \ldots, x_N) \assign \left\{\begin{array}{l}
       - qm_{k - 1}^{(q)} (x_1, \ldots, x_N) \text{\quad for } k \geq 1\\
       1 \text{\quad for } k = 0,
     \end{array}\right. \]
  implies that $h_k^{(+ q)} (x_1, \ldots, x_N)$ is equal to the
  super-symmetric complete homogeneous polynomial of order $k$, in two sets of
  variables $(x_1, \ldots, x_N), (y_1, \ldots, y_N)$, if we consider $y_i = -
  (x_i + q)$ for every $i = 1, \ldots, N$.
\end{remark}

There is a straightforward procedure in the literature that allows the passage
from Schur generating functions to $q$-deformed Perelomov-Popov measures. This
can be done by applying certain differential operators on the Schur generating
functions. Such operators were firstly introduced in {\cite{B1}} for the cases
$q = 0, 1$ where they realized that the $k$-th moments of the corresponding
$q$-deformed Perelomov-Popov are eigenvalues of the character of the
irreducible representation of $U (N)$, of the corresponding signature.

We now introduce a differential operator that depends on $q \in [- 1, 1]$ and
acts on functions $f$ of $N$ variables that are analytic:
\[ \mathcal{D}_k^{U (N), q} (f) (u_1, \ldots, u_N) \assign \left( \prod_{i <
   j} \frac{1}{u_i - u_j} \right)  \sum_{i = 1}^N u_i^{1 - q} \partial_i
   \left[ (u_i \partial_i)^{k - 1} \left( \prod_{i < j} (u_i - u_j) f (u_1,
   \ldots, u_N) \right) \right] . \]
To make our notation clear, $\partial_i$ denotes the partial derivative with
respect to the variable $u_i$ and $(u_i \partial_i) f (u_1, \ldots, u_N)
\assign u_i \partial_i f (u_1, \ldots, u_N)$ for every $i = 1, \ldots, N$.

\begin{proposition}
  \label{p9}For every $\mathlambda (N) = (\mathlambda_1 (N) \geq \cdots \geq
  \mathlambda_N (N)) \in \hat{U} (N)$ and $k \in \mathbb{N}$ we have
  \begin{equation}
    \frac{1}{\chi^{\mathlambda (N)} (1^N)} \mathcal{D}_k^{U (N), q}
    \chi^{\mathlambda (N)} (u_1, \ldots, u_N) \longdownminus_{u_1 = \cdots =
    u_N = 1} = N^{k + 1} \int_{\mathbb{R}} t^k m_{N, P \text{} P (q)}
    [\mathlambda (N)] (d \text{} t) . \label{exiswsi6}
  \end{equation}
\end{proposition}

\begin{proof}
  Using the Leibniz formula for the determinant, we immediately get that
  \[ \mathcal{D}_k^{U (N), q} \chi^{\mathlambda (N)} (u_1, \ldots, u_N) =
     \sum_{i = 1}^N (\mathlambda_i (N) + N - i)^k \chi^{\mathlambda_i^{(- q)}
     (N)} (u_1, \ldots, u_N), \]
  where $\mathlambda_i^{(- q)} (N) \assign (\mathlambda_1 (N), \ldots,
  \mathlambda_{i - 1} (N), \mathlambda_i (N) - q, \mathlambda_{i + 1} (N),
  \ldots, \mathlambda_N (N))$. We will denote the $j$-th coordinate of
  $\mathlambda_i^{(- q)} (N)$ by $\mathlambda^{(- q)}_{i, j} (N)$. Of course
  $\mathlambda_i^{(- q)} (N)$ is not a signature, but an $N$-tuple of reals
  which are decreasing for $q \in (- 1, 1)$ but not necessarily for $q = \pm
  1$. For the signature $\mathlambda (N)$, the Weyl's dimension formula
  implies that
  \[ \chi^{\mathlambda (N)} (1^N) = \prod_{1 \leq \alpha < \beta \leq N}
     \frac{(\mathlambda_{\alpha} (N) - \alpha) - (\mathlambda_{\beta} (N) -
     \beta)}{\beta - \alpha} . \]
  The value of $\chi^{\mathlambda_i^{(- q)} (N)}$ at $1^N$ can be computed
  evaluating the multivariate Bessel function
  \[ B_{\mathlambda_i^{(- q)} (N)} (u_1, \ldots, u_N) \assign \frac{\det
     (e^{u_{\alpha}  (\mathlambda_{i, \beta}^{(- q)} (N) + N - \beta)})_{1
     \leq \alpha, \beta \leq N}}{\prod_{1 \leq \alpha < \beta \leq N}
     (u_{\alpha} - u_{\beta})}, \]
  at $0^N$. The Harish-Chandra integral formula (see in {\cite{B24}}) implies
  that this value is equal to
  \[ \prod_{1 \leq \alpha < \beta \leq N} \frac{(\mathlambda_{i, \alpha}^{(-
     q)} (N) - \alpha) - (\mathlambda_{i, \beta}^{(- q)} (N) - \beta)}{\beta -
     \alpha} . \]
  This proves the claim.
\end{proof}

\subsection{Asymptotic behavior and moments}

Now, we focus on the asymptotic study of the random measure $m_{N, P \text{} P
(q)} [\mathlambda (N)]$ where $\mathlambda (N) \in \hat{U} (N), N \in
\mathbb{N}$ are random signatures such that the Schur generating functions that
corresponds to the distribution of $\mathlambda (N)$, satisfy certain
asymptotics as $N \rightarrow \infty$. Our goal is to extract a Law of Large
numbers for $m_{N, P \text{} P (q)} [\mathlambda (N)]$.

In Proposition \ref{p9} we described the procedure that gives us the moments
of the $q$-deformed Perelomov-Popov measures. The control of these moments as
$N \rightarrow \infty$ relies on the way that the differential operator
$\mathcal{D}_k^{U (N), q}$ acts on analytic functions of $N$ variables. This
action can be described explicitly when our analytic function is symmetric. Of
course this is a property that the Schur generating function satisfies.

\begin{theorem}
  \label{12}Let $\varrho (N), N \in \mathbb{N}$, be a sequence of probability
  measures on $\hat{U} (N)$. Moreover we assume that there exists a function
  $\Psi$, analytic in a neighborhood of $1$ such that for every $r \in
  \mathbb{N}$ fixed,
  \begin{equation}
    \lim_{N \rightarrow \infty} \frac{1}{N} \log \text{} S_{\varrho (N)} (u_1,
    \ldots, u_r, 1^{N - r}) = \sum_{i = 1}^r \Psi (u_i), \label{assumption}
  \end{equation}
  where the convergence is uniform in a neighborhood of $1^r$. Then the random
  measure $m_{N, P \text{} P (q)} [\varrho (N)]$ converges as $N \rightarrow
  \infty$ in probability, in the sense of moments to a deterministic measure
  $\tmmathbf{\mu}^{(q)}$ on $\mathbb{R}$, with moments
  $(\tmmathbf{\mu}_k^{(q)})_{k \in \mathbb{N}}$ given by
  \begin{equation}
    \tmmathbf{\mu}_k^{(q)} = \sum_{m = 0}^k \binom{k}{m} \frac{1}{(m + 1) !} 
    \left. \frac{\mathd^m}{\mathd u^m} (u^{k - q} (\Psi' (u))^{k - m})
    \right|_{u = 1} . \label{moments}
  \end{equation}
\end{theorem}

\begin{proof}
  First, we show convergence in expectation, i.e. for every $k \in \mathbb{N}$
  \begin{equation}
    \lim_{N \rightarrow \infty} \mathbb{E}_{\mathlambda (N)} \left[
    \int_{\mathbb{R}} t^k m_{N, P \text{} P (q)} [\mathlambda (N)] \left( d
    \text{} t \right) \right] =\tmmathbf{\mu}_k^{(q)} . \label{expectation}
 \end{equation}
  For this purpose we will show that there is an expansion of the form $N^{k +
  1} M^{(q)}_{0, k, N} (\Psi) + N^k M^{(q)}_{1, k, N} (\Psi) + \cdots +
  M^{(q)}_{k + 1, k, N} (\Psi)$ for $\mathcal{D}_k^{U (N), q} S_{\varrho (N)}
  \longdownminus_{u_1 = \cdots = u_N = 1}$, where $(M_{i, k, N}
  (\Psi))_{N \in \mathbb{N}}$ converge and $\lim_{N \rightarrow \infty}
  M^{(q)}_{0, k, N} (\Psi) =\tmmathbf{\mu}_k^{(q)}$. For the cases $q = 0, 1$,
  it was studied extensively in {\cite{B1}}, {\cite{B41}}, how the assumption
  (\ref{assumption}) leads to such an expansion and which are the limits of
  the corresponding sequences. For the general case $q \in [- 1, 1]$, similar
  arguments hold. Below we briefly describe this procedure and how the general
  $q \in [- 1, 1]$ affects the limit of the coefficient of $N^{k + 1}$. For
  more details we refer to the above papers.
  
  By definition, the function $\mathcal{D}_k^{U (N), q} S_{\varrho (N)} (u_1,
  \ldots, u_N)$ can be written as a linear combination of terms
  \begin{equation}
    \left( \prod_{1 \leq i < j \leq N} \frac{1}{u_i - u_j} \right)  \sum_{i =
    1}^N u_i^{n - q} \partial_i^n \left( \prod_{1 \leq i < j \leq N} (u_i -
    u_j) S_{\varrho (N)} (u_1, \ldots, u_N) \right), \label{8}
  \end{equation}
  for $n = 1, \ldots, k$. The leading term of the expansion $N^{k + 1} M_{0,
  k, N}^{(q)} (\Psi)$ emerges from the term (\ref{8}) for $n = k$, which has
  coefficient $1$. From Lemma 1 of {\cite{B41}}, it emerges that this is equal
  to
  \begin{equation}
    \sum_{m = 0}^k \sum_{\underset{l_i \neq l_j \text{ for } i \neq j}{l_0,
    \ldots, l_m = 1}}^N \binom{k}{m}  \frac{u_{l_0}^{k - q} \partial_{l_0}^{k
    - m} S_{\varrho (N)} (u_1, \ldots, u_N)}{(u_{l_0} - u_{l_1}) \ldots
    (u_{l_0} - u_{l_m})} . \label{9}
  \end{equation}
  Assumption (\ref{assumption}) indicates to write $S_{\varrho (N)} = \exp
  \left( N \cdot \frac{1}{N} \log \text{} S_{\varrho (N)} \right)$ before we
  differentiate, since we want to study (\ref{9}) as $u_1 = \cdots = u_N = 1$
  and $N \rightarrow \infty$. Doing this we have that (\ref{9}) is a linear
  combination of terms
  \[ \frac{u_{b_0}^{k - q}  \left( \partial_{b_0} \left( \frac{1}{N} \log
     \text{} S_{\varrho (N)} \right) \right)^{l_1} \ldots \left( \partial^{k -
     m}_{b_0} \left( \frac{1}{N} \log \text{} S_{\varrho (N)} \right)
     \right)^{l_{k - m}}}{(u_{b_0} - u_{b_1}) \ldots (u_{b_0} - u_{b_m})} +
     \text{ {\hspace{11em}}} \]
  \begin{equation}
    \text{{\hspace{8em}}} \quad \ldots + \frac{u_{b_m}^{k - q} \left(
    \partial_{b_m} \left( \frac{1}{N} \log \text{} S_{\varrho (N)} \right)
    \right)^{l_1} \ldots \left( \partial_{b_m}^{k - m} \left( \frac{1}{N} \log
    \text{} S_{\varrho (N)} \right) \right)^{l_{k - m}}}{(u_{b_m} - u_{b_0})
    \ldots (u_{b_m} - u_{b_{m - 1}})}, \label{esta}
  \end{equation}
  where $l_1 + 2 l_2 + \cdots + (k - m) l_{k - m} = k - m$ and $b_0, \ldots,
  b_m \in \{1, \ldots, N\}$ are distinct. For $u_1 = \cdots = u_N = 1$, we have
  that (\ref{esta}) converges (see e.g. Theorem 8 of {\cite{B41}}), the limit
  does not depend on $b_0, \ldots, b_m$ and due to (\ref{assumption}) the
  limit will converge as $N \rightarrow \infty$. For these reasons we can get
  the desirable expansion for $\mathcal{D}_k^{U (N), q} S_{\varrho (N)}
  \longdownminus_{u_1 = \cdots = u_N = 1}$. The leading term emerges from the
  terms (\ref{esta}) for $l_1 = k - m, l_2 = \cdots = l_{k - m} = 0$ and it
  is equal to
  \[ M_{0, k, N}^{(q)} (\Psi) = \sum_{m = 0}^k \binom{k}{m} \frac{1}{(m + 1)
     !} \partial_1^m \left( u_1^{k - q} \left( \partial_1 \left( \frac{1}{N}
     \log \text{} S_{\varrho (N)} \right) \right)^{k - m} \right) (1^N) +
     \omicron (1) . \]
  Thus the relation (\ref{expectation}) holds.
  
  In order to show convergence in probability note that by Proposition
  \ref{Prop9} all the moments of $m_{N, P \text{} P (q)} [\varrho (N)]$
  converge in probability if all the moments of $m_{N, P \text{} P (0)}
  [\varrho (N)]$ converge in probability as well. The convergence in
  probability for the moments of $m_{N, P \text{} P (0)} [\varrho (N)]$ was
  proved in {\cite{B1}} where basically they proved the more general result
  \begin{equation}
    \lim_{N \rightarrow \infty} \mathbb{E}_{\mathlambda (N)} \left[ \left(
    \frac{1}{N} \sum_{i = 1}^N \left( \frac{\mathlambda_i (N) + N - i}{N}
    \right)^k \right)^l \right] = (\tmmathbf{\mu}_k^{(0)})^l \text{\quad for
    every } k, l \in \mathbb{N}, \label{VAp}
  \end{equation}
  applying recursively the differential operator $\mathcal{D}_k^{U (N), 0}$ on
  $S_{\varrho (N)}$. Therefore (\ref{VAp}) implies that the limit in
  probability of the $k$-th moment of $m_{N, P \text{} P (q)} [\varrho (N)]$
  is equal to the limit in expectation.
\end{proof}

\subsection{Asymptotic behavior and $R$-transform}

In Theorem \ref{12} we showed how from a certain asymptotic behavior for the
logarithm of the Schur generating functions of $(\varrho (N))_{N \in
\mathbb{N}}$, we can extract a Law of Large Numbers for $m_{N, P \text{} P
(q)} [\varrho (N)]$. It is clear that in this case, the moments of the
limiting measure are uniquely determined by the function $\Psi$ of
(\ref{assumption}), however from the formula (\ref{moments}) for the moments,
it is difficult to recognize in which way $\Psi$ characterizes the limiting
measure. For the cases $q = 0, 1$ this problem have been solved in
{\cite{B1}}. In these cases the function $\Psi$ is essential in order to
describe the $R$-transforms of the corresponding limiting measures.

Let $R^{(0)}$ and $R^{(1)}$ be the $R$-transforms of the limiting measures
$\tmmathbf{\mu}^{(0)}, \tmmathbf{\mu}^{(1)}$ of Theorem \ref{12}. Then it is
known (see in {\cite{B1}}) that these functions are given by
\begin{equation}
  R^{(0)} (z) = e^z \Psi' (e^z) + \frac{e^z}{e^z - 1} - \frac{1}{z}
  \text{\quad and\quad} R^{(1)} (z) = \frac{1}{1 - z} \Psi' \left( \frac{1}{1
  - z} \right), \label{16}
\end{equation}
for $|z|$ small enough. These relations clarify completely the kind of
dependence of the limiting measures from the function $\Psi$. Also, the
formula for $R^{(0)}$ motivates the definition of the quantized free
convolution. Our goal is to investigate how for general $q \in [- 1, 1]$, the
$R$-transform $R^{(q)}$ of $\tmmathbf{\mu}^{(q)}$ can be expressed in terms of
$\Psi$.

Looking at the formula for the moments $\tmmathbf{\mu}_k^{(q)}$, $q \in [- 1,
1]$, and using the moment-cumulant relations (\ref{momentcumulantrelations})
in order to compute the free cumulants $\kappa_n (\tmmathbf{\mu}^{(q)})$, for
$n$ small, it seems that for the family of $q$-deformed Perelomov-Popov
measures we have the phenomenon that $\kappa_n (\tmmathbf{\mu}^{(q)})$ is a
sum of two terms where the first summand is linear on $\Psi$ and the second
summand does not depend on $\Psi$. For example
\[ \kappa_1 (\tmmathbf{\mu}^{(q)}) = \Psi' (1) + \frac{1 - q}{2} \text{\quad
   and\quad} \kappa_2 (\tmmathbf{\mu}^{(q)}) = \Psi'' (1) + \Psi' (1) +
   \frac{1 - q^2}{12} . \]
Consequently, the same thing holds for $R^{(q)}$. The formulas in (\ref{16})
for $q = 0, 1$ makes us to expect that the summand of $R^{(q)} (z)$ that
depends on $\Psi$ will have the form $e_q (z) \Psi' (e_q (z))$ where $e_q (z)
= \sum_{n \geq 0} \frac{1}{n^{(q)}} z^n$ in a neighborhood of $0$ and
$(n^{(q)})_{n \in \mathbb{N}}$ is a sequence of reals. Of course $n^{(0)} =
n!$, $n^{(1)} = 1$ for every $n \in \mathbb{N}$ and for $q \in [- 1, 1]$
taking into account the formula (\ref{moments}) for $\tmmathbf{\mu}_k^{(q)}$,
one should expect that
\[ n^{(q)} = \frac{n!}{((n - 1) q + 1) ((n - 2) q + 1) \ldots (q + 1)} \text{
   for every } n \in \mathbb{N} \text{ and } e_q (z) = (1 - q \text{} z)^{-
   \frac{1}{q}} . \]
In order to determine the summand of $R^{(q)} (z)$ that does not depend on
$\Psi$, we just have to compute the $R$-transform of $\tmmathbf{\mu}^{(q)}$
for $\Psi \equiv 0$. In that case $\tmmathbf{\mu}_k^{(q)} = \frac{1}{(k + 1)
!} \prod_{i = 1}^k (i - q)$, which is the $k$-th moment of $\beta (1 - q, 1 +
q)$. Thus, the above observations make us believe that
\begin{equation}
  R^{(q)} (z) = e_q (z) \Psi' (e_q (z)) + R_{\beta (1 - q, 1 + q)} (z) = e_q
  (z) \Psi' (e_q (z)) + \frac{e_q (z)}{e_q (z) - 1} - \frac{1}{z},
  \label{xana}
\end{equation}
for $|z|$ small enough. In this subsection we will prove this claim.

We will use a different approach in order to show convergence of $m_{N, P
\text{} P (q)} [\varrho (N)]$, as $N \rightarrow \infty$. Although our
strategy remains the same, we have to modify the differential operator
$\mathcal{D}_k^{U (N), q}$ and to replace the Schur generating functions by
another sequence of symmetric functions of $N$ variables, analytic in a
neighborhood of $0^N$. We will rely on the fact that we can immediately
extract the $R$-transform of a probability measure on $\mathbb{R}$ when its
moments are given by certain formulas (see Lemma 4 of {\cite{B41}}). The same
strategy was followed in {\cite{B41}} in order to give an alternate proof for
the relation in (\ref{16}) concerning $R^{(0)}$.

\begin{definition}
  Let $\varrho (N)$ be a sequence of probability measures on $\hat{U} (N)$, $N
  \in \mathbb{N}$. For $q \in [- 1, 1]$, we define the $q$-deformed Schur
  generating functions with respect to $\varrho (N)$:
  \[ T_{\varrho (N)}^{(q)} (u_1, \ldots, u_N) \assign \left( \prod_{1 \leq i <
     j \leq N} \frac{e_q (u_i) - e_q (u_j)}{u_i - u_j} \right) S_{\varrho (N)}
     (e_q (u_1), \ldots, e_q (u_N)) . \]
\end{definition}

Of course, $T_{\varrho (N)}^{(q)}$ is a symmetric function, well defined in a
neighborhood of $0^N$. Assuming that (\ref{assumption}) holds, this determines
the asymptotics of the normalized logarithm of $T_{\varrho (N)}^{(q)}$. The
next step is to introduce the differential operator that we will apply to
$T_{\varrho (N)}^{(q)}$: For $k \in \mathbb{N}$ and an analytic function $f$
of $N$ variables, define
\[ \mathfrak{D}_k^{U (N), q} (f) (u_1, \ldots, u_N) \assign \left( \prod_{1
   \leq i < j \leq N} \frac{1}{u_i - u_j} \right)  \sum_{i = 1}^N (1 -
   qu_i)^{k + 1} \partial_i^k \left( \prod_{1 \leq i < j \leq N} (u_i - u_j) f
   (u_1, \ldots, u_N) \right) . \]
\begin{proposition}
  \label{p14}For every $\mathlambda (N) = (\mathlambda_1 (N) \geq \cdots \geq
  \mathlambda_N (N)) \in \hat{U} (N)$ and $k \in \mathbb{N}$ we have
  \[ \frac{1}{\chi^{\mathlambda (N)} (1^N)}  \left. \mathfrak{D}_k^{U (N), q}
     \left( \frac{\det (e_q (u_i)^{\mathlambda_j (N) + N - j})_{1 \leq i, j
     \leq N}}{\prod_{1 \leq i < j \leq N} (u_i - u_j)} \right) \right|_{u_1 =
     \cdots = u_N = 0} \]
  \begin{equation}
    \text{{\hspace{10em}}} = \sum_{i = 1}^N \left( \prod_{j \neq i}
    \frac{(\mathlambda_i (N) - i) - (\mathlambda_j (N) - j) -
    q}{(\mathlambda_i (N) - i) - (\mathlambda_j (N) - j)} \right) \prod_{j =
    0}^{k - 1} (\mathlambda_i (N) + N - i + j \text{} q) .
  \end{equation}
\end{proposition}

\begin{proof}
  Same argument as in Proposition \ref{p9}.
\end{proof}

\begin{remark}
  Note that the operator $\mathfrak{D}_k^{U (N), 0}$ is significantly simpler
  from the operators $\mathfrak{D}_k^{U (N), q}$, $q \neq 0$. This operator is
  crucial for the study of the empirical spectral distribution of large
  unitarily invariant random matrices. In this context, it was studied in
  {\cite{B41}}, {\cite{B8}}. One of its important properties is that for an
  arbitrary sequence $f_N$ of analytic and symmetric functions of $N$
  variables such that for every $k$ finite $\left( (u_1, \ldots, u_k) \mapsto
  \frac{1}{N} \log \text{} f_N (u_1, \ldots, u_k, 0^{N - k}) \right)_{N \in
  \mathbb{N}}$ converges uniformly in an additive symmetric function in a
  neighborhood of $0^k$, we can determine explicitly the non-crossing
  partition expansion of $\lim_{N \rightarrow \infty} N^{- k - 1}
  \mathfrak{D}_k^{U (N), 0} (f_N) |_{u_1 = \cdots = u_N = 0}$. This means that
  we can determine the unique sequence $(\alpha_n)_{n \in \mathbb{N}}$ of
  reals for which we have that for every $k \in \mathbb{N}$ the above limit is
  equal to $\sum_{\pi \in \tmop{NC} (k)} \prod_{V \in \pi} \alpha_{|V|}$.
  That's the way that it was proved in {\cite{B41}} the formula of relation
  (\ref{16}) for the $R$-transform of the limiting measure
  $\tmmathbf{\mu}^{(q)}$, of Theorem \ref{12}, for $q = 0.$ For the case $q
  \neq 0$ the determination of the $R$-transform is a slightly more
  complicated process.
\end{remark}

\begin{theorem}
  \label{kolpa}Let $\varrho (N)$, $N \in \mathbb{N}$, be a sequence of
  probability measures on $\hat{U} (N)$. Moreover we assume that there exists
  a function $\Psi$, analytic in a neighborhood of $1$ such that
  \begin{equation}
    \lim_{N \rightarrow \infty} \frac{1}{N} \log \text{} S_{\varrho (N)} (u_1,
    \ldots, u_r, 1^{N - r}) = \sum_{i = 1}^r \Psi (u_i), \label{main}
  \end{equation}
  where the convergence is uniform in a neighborhood of $1^r$. Then for every
  $k \in \mathbb{N}$, the $k$-th moment of $m_{N, P \text{} P (q)} [\varrho
  (N)]$ converges as $N \rightarrow \infty$ in probability to the sum
  \begin{equation}
    \sum_{m = 0}^{k - 1} \frac{k!}{m! (m + 1) ! (k - m) !}  \left.
    \frac{\mathd^m}{\mathd u^m} \left( e_q (u) \Psi' (e_q (u)) + \frac{e_q
    (u)}{e_q (u) - 1} - \frac{1}{u} \right)^{k - m} \right|_{u = 0} .
    \label{19}
  \end{equation}
\end{theorem}

\begin{proof}
  Taking into account Theorem \ref{12}, it suffices to show convergence of the
  $k$-th moment of $\mathbb{E} [m_{N, P \text{} P (q)} [\varrho (N)]]$ to
  (\ref{19}). The assumption (\ref{main}) implies that
  \begin{equation}
    \lim_{N \rightarrow \infty} \frac{1}{N} \log T^{(q)}_{\varrho (N)} (u_1,
    \ldots, u_r, 0^{N - r}) = \sum_{i = 1}^r \Psi (e_q (u_i)) + \log \left(
    \frac{e_q (u_i) - 1}{u_i} \right), \label{conv}
  \end{equation}
  uniformly in a neighborhood of $0^r$. The strategy of our proof is the same
  as in Theorem \ref{12}, namely we will apply the differential operator
  $\mathfrak{D}_k^{U (N), q}$ to $T^{(q)}_{\varrho (N)}$ and set $u_1 = \cdots
  = u_N = 0$ in order to get an expansion of the form $N^{k + 1} M_{0, k,
  N}^{(q)} (\Psi) + N^k M_{1, k, N}^{(q)} (\Psi) + \cdots + M_{k + 1, k,
  N}^{(q)} (\Psi)$, for
  \begin{equation}
    \mathbb{E}_{\mathlambda (N)} \left[ \sum_{i = 1}^N \left( \prod_{j \neq i}
    \frac{(\mathlambda_i (N) - i) - (\mathlambda_j (N) - j) -
    q}{(\mathlambda_i (N) - i) - (\mathlambda_j (N) - j)} \right) \prod_{j =
    0}^{k - 1} (\mathlambda_i (N) + N - i + j \text{} q) \right] . \label{21}
  \end{equation}
  The sequences $(M_{i, k, N}^{(q)} (\Psi))_{N \in \mathbb{N}}$ will converge.
  Also since we have already proved that (\ref{main}) implies convergence for
  the moments of $m_{N, P \text{} P (q)} [\varrho (N)]$, a suitable expansion
  for (\ref{21}) suffices in order to prove the claim.
  
  Using the same arguments as in Lemma 1 of {\cite{B41}} we have that
  \begin{equation}
    \mathfrak{D}_k^{U (N), q} (T^{(q)}_{\varrho (N)}) (u_1, \ldots, u_N) =
    \sum_{m = 0}^k \sum_{\underset{l_i \neq l_j \text{ for } i \neq j}{l_0,
    \ldots, l_m = 1}}^N \binom{k}{m}  \frac{(1 - qu_{l_0})^{k + 1}
    \partial_{l_0}^{k - m} T^{(q)}_{\varrho (N)} (u_1, \ldots, u_N)}{(u_{l_0}
    - u_{l_1}) \ldots (u_{l_0} - u_{l_m})}, \label{22}
  \end{equation}
  and we write $T^{(q)}_{\varrho (N)} = \exp \left( N \cdot \frac{1}{N} \log
  T^{(q)}_{\varrho (N)} \right)$ in order to get the desirable expansion.
  Then, the reasons that lead to the expansion that we claimed are the same as
  in Theorem \ref{12} and they are explained in detail in Theorem 8 of
  {\cite{B41}}. Since at this moment we are interested in $\lim_{N \rightarrow
  \infty} M_{0, k, N}^{(q)} (\Psi)$, we concentrate on how from equality
  (\ref{22}) we can compute $M_{0, k, N}^{(q)} (\Psi)$.
  
  Applying the chain rule for the derivatives $\partial_{l_0}^{k - m} \exp
  \left( N \cdot \frac{1}{N} \log T^{(q)}_{\varrho (N)} \right)$, the only
  summands that will contribute to $N^{k + 1} M_{0, k, N}^{(q)} (\Psi)$ are
  $N^{k - m} T^{(q)}_{\varrho (N)} \left( \partial_{l_0} \left( \frac{1}{N}
  \log T^{(q)}_{\varrho (N)} \right) \right)^{k - m}$. This is the case
  because the summands
  \[ \binom{k}{m} N^{k - m} \frac{(1 - qu_{b_0})^{k + 1} \left( \partial_{b_0}
     \left( \frac{1}{N} \log T^{(q)}_{\varrho (N)} \right) \right)^{k -
     m}}{(u_{b_0} - u_{b_1}) \ldots (u_{b_0} - u_{b_m})} +
     \text{{\hspace{17em}}} \]
  \begin{equation}
    \text{{\hspace{10em}}} \ldots + \binom{k}{m} \left. N^{k - m} \frac{(1 -
    qu_{b_m})^{k + 1} \left( \partial_{b_m} \left( \frac{1}{N} \log
    T^{(q)}_{\varrho (N)} \right) \right)^{k - m}}{(u_{b_m} - u_{b_0}) \ldots
    (u_{b_m} - u_{b_{m - 1}})} \right|_{u_1 = \cdots = u_N = 0}, \label{fwta}
  \end{equation}
  of $\mathfrak{D}_k^{U (N), q} T^{(q)}_{\varrho (N)} |_{u_1 = \cdots = u_N =
  0}$, do not depend on $b_0, \ldots, b_m \in \{1, \ldots, N\}$. The explicit
  computation of (\ref{fwta}) is a far complicated procedure, however due to
  (\ref{conv}) we are interested in its summands that include only derivatives
  where we differentiate with respect to only one variable. These can be
  clarified easily due to the results of {\cite{B1}}, {\cite{B41}}.
  
  Summarizing, by the above we have that
  \[ M_{0, k, N}^{(q)} (\Psi) = \sum_{m = 0}^k \frac{k!}{m! (m + 1) ! (k - m)
     !} \partial_1^m \left( (1 - qu_1)^{k + 1} \left( \partial_1 \left(
     \frac{1}{N} \log T^{(q)}_{\varrho (N)} \right) \right)^{k - m} \right)
     (0^N) + o (1), \]
  and since $e_q' (u) = e_q^{q + 1} (u)$, this converges to
  \begin{equation}
    \sum_{m = 0}^k \frac{k!}{m! (m + 1) ! (k - m) !}  \frac{\mathd^m}{\mathd
    u^m} \left. \left( \frac{1}{e_q^{(m + 1) q} (u)} \left( e_q (u) \Psi' (e_q
    (u)) + \frac{e_q (u)}{e_q (u) - 1} - \frac{1}{u \text{}} + q \right)^{k -
    m} \right) \right|_{u = 0} . \label{j}
  \end{equation}
  In order to make it clear why (\ref{j}) is equal to (\ref{19}) we consider
  \[ F_q (u) \assign e_q (u) \Psi' (e_q (u)) + \frac{e_q (u)}{e_q (u) - 1} -
     \frac{1}{u} . \]
  Then, the $m$-th order derivative in (\ref{j}) is equal to
  \[ \sum_{n = 0}^{k - m} \sum_{l = 0}^m \binom{k - m}{n} \binom{m}{l}
     \frac{(m + 1) !}{(l + 1) !} q^n (- q)^{m - l}  \left.
     \frac{\mathd^l}{\mathd u^l} F_q^{k - m - n} (u) \right|_{u = 0} . \]
  As a corollary, (\ref{j}) is equal to
  \[ \sum_{l = 0}^k \sum_{n = 0}^{k - l} \sum_{m = l}^{k - n} \frac{k!}{l! (l
     + 1) ! (k - l) !} \binom{k - l}{n} \binom{k - l - n}{m - l} q^n (- q)^{m
     - l}  \left. \frac{\mathd^l}{\mathd u^l} F_q^{k - m - n} (u) \right|_{u =
     0} \]
  \[ = \sum_{l = 0}^k \sum_{n = 0}^{k - l} \frac{k!}{l! (l + 1) ! (k - l) !}
     \binom{k - l}{n} q^n  \left. \frac{\mathd^l}{\mathd u^l} (F_q (u) - q)^{k
     - l - n} \right|_{u = 0} \]
  \[ = \sum_{l = 0}^k \frac{k!}{l! (l + 1) ! (k - l) !}  \left.
     \frac{\mathd^l}{\mathd u^l} F_q^{k - l} (u) \right|_{u = 0} . \]
\end{proof}

\begin{remark}
  Comparing Theorem \ref{12} and Theorem \ref{kolpa} we see that if the
  relation (\ref{main}) is satisfied for a sequence $\varrho (N)$ of
  probability measures on $\hat{U} (N)$, $N \in \mathbb{N}$, then (\ref{19})
  must be equal to (\ref{moments}). The significance of formula (\ref{19}) for
  the moments is that it immediately gives us the $R$-transform of the
  limiting measure $\tmmathbf{\mu}^{(q)}$. That's the upshot of Lemma 4 of
  {\cite{B41}}, from which we deduce that our claim for the $R$-transform
  $R^{(q)}$ holds, i.e. it is given by formula (\ref{xana}).
\end{remark}

\begin{remark}
  \label{Remark}In the framework of Theorem \ref{kolpa}, obviously the
  $1$-deformed Perelomov-Popov measure joins the regime of free probability,
  in the sense that $R^{(1)}$ is linear with respect to $\Psi$. Under slight
  modifications the same holds true for the second boundary case $q = - 1$.
  More precisely for the shifted by $1$ Perelomov-Popov measure
  \[ m^{(- 1)}_{N, P \text{} P (- 1)} [\varrho (N)] \assign \frac{1}{N}
     \sum_{i = 1}^N \left( \prod_{j \neq i} \frac{(\mathlambda_i (N) - i) -
     (\mathlambda_j (N) - j) + 1}{(\mathlambda_i (N) - i) - (\mathlambda_j (N)
     - j)} \right) \delta \left( \frac{\mathlambda_i (N) - i}{N} \right), \]
  where $\mathlambda (N) \in \hat{U} (N)$ is random with distribution $\varrho
  (N)$, Theorem \ref{kolpa} implies that it converges to a probability measure
  with $R$-transform $z \mapsto (z + 1) \Psi' (z + 1)$.
\end{remark}

\subsection{Extreme Characters and $q$-deformed Perelomov-Popov
measures}\label{Y}

In this subsection we recall some connections between Schur generating
functions and extreme characters of the infinite unitary group. We recall that
for $q = - 1, 1$, the $q$-deformed quantized free convolution is the classical
free convolution, which is related to random matrices. For these particular
cases, classical probability distributions strongly related to random
matrices, are limits of $q$-deformed Perelomov-Popov measures, connected to
extreme characters of $U (\infty)$.

We denote by $U (\infty)$ the group of infinite unitary matrices $U = (U_{i,
j})_{i, j \in \mathbb{N}}$ with finitely many entries $U_{i, j}$ different
from $\delta_{i, j}$. An important class of functions on $U (\infty)$ are
characters, which are central normalized positive-definite functions on $U
(\infty)$ that are in one-to-one correspondence with the spherical unitary
representations, see e.g. in {\cite{B29}}, {\cite{B28}}. Given a character
$\tmmathbf{\chi}: U (\infty) \rightarrow \mathbb{C}$, there is a sequence of
probability measures $\varrho (N)$ on $\hat{U} (N)$, $N \in \mathbb{N}$, such
that $\tmmathbf{\chi} (\tmop{diag} (u_1, \ldots, u_N, 1, 1, \ldots)) =
S_{\varrho (N)} (u_1, \ldots, u_N)$.

A special role from the side of representation theory play the extreme points
of the convex set of characters of $U (\infty)$, which are known as extreme
characters. The classification of extreme characters emerges from the
Edrei-Voiculescu theorem (see e.g. {\cite{B30}}, {\cite{B31}}). According to
this theorem, an arbitrary extreme character $\tmmathbf{\chi}: U (\infty)
\rightarrow \mathbb{C}$ has a multiplicative form, i.e.
\begin{equation}
  \tmmathbf{\chi} (U) = \prod_{u \in \tmop{Spectrum} (U)} \Phi_{\omega} (u),
  \text{\quad for every\quad} U \in U (\infty), \label{extremecharacters}
\end{equation}
where $\Phi_{\omega} : \{u \in \mathbb{C} \of |u| = 1\} \rightarrow
\mathbb{C}$ is a complex function such that $\Phi_{\omega} (1) = 1$. To be
more precise, the function $\Phi_{\omega}$ is parametrized by a family of
variables $\omega = (\{\alpha_i^+ \}_{i \geq 0}, \{\alpha_i^- \}_{i \geq 0},
\{\beta_i^+ \}_{i \geq 0}, \{\beta_i^- \}_{i \geq 0}, \delta^+, \delta^-)$
where all the sequences in $\omega$ are non-increasing sequences of
non-negative reals and $\delta^+, \delta^- \geq 0$. Moreover we have that
\[ \sum_{i = 1}^{\infty} (\alpha_i^{\pm} + \beta_i^{\pm}) \leq \delta^{\pm}
   \text{\quad and\quad} \beta_1^+ + \beta_1^- \leq 1. \]
For such an $\omega$, the function $\Phi_{\omega}$ is given by
\[ \Phi_{\omega} (u) = \exp (\gamma^+ (u - 1) + \gamma^- (u^{- 1} - 1))
   \prod_{i = 1}^{\infty} \frac{(1 + \beta_i^+ (u - 1)) (1 + \beta_i^- (u^{-
   1} - 1))}{(1 + \alpha_i^+ (u - 1)) (1 - \alpha_i^- (u^{- 1} - 1))}, \]
where $\gamma^{\pm} \assign \delta^{\pm} - \sum_{i \geq 0} (\alpha_i^{\pm} +
\beta_i^{\pm}) \geq 0$. The function $\Phi_{\omega}$ is known as Voiculescu
function with parameter $\omega$.

As a corollary of relation (\ref{extremecharacters}), the logarithm of Schur
generating functions of probability measures on $\hat{U} (N)$ that correspond
to extreme characters, is an additive function for every $N \in \mathbb{N}$
and not just asymptotically. Therefore, Theorem \ref{12} and Theorem
\ref{kolpa} can be applied in order to determine the limiting measures of
$q$-deformed Perelomov-Popov measures of signatures that are distributed
according to the above probability measures on $\hat{U} (N)$. For example, let
$\tmmathbf{\chi}_N^{\gamma} : U (\infty) \rightarrow \mathbb{C}$ be the
extreme characters that correspond to the Voiculescu functions
$\Phi_N^{\gamma} (u) = \exp (N \gamma (u - 1))$. Then we can consider the
probability measures $\varrho^{\gamma} (N)$ on $\hat{U} (N)$, $N \in
\mathbb{N}$, that satisfy the relation
\begin{equation}
  S_{\varrho^{\gamma} (N)} (u_1, \ldots, u_N) =\tmmathbf{\chi}_N^{\gamma}
  (\tmop{diag} (u_1, \ldots, u_N, 1, 1, \ldots)) = \prod_{i = 1}^N
  \Phi_N^{\gamma} (u_i), \label{dani}
\end{equation}
for every $N \in \mathbb{N}$ and $u_1, \ldots, u_N \in \{u \in \mathbb{C} \of
|u| = 1\}$. For $q = 0$, it was proved in {\cite{B42}}, {\cite{B2}} that the
probability density function of the corresponding limiting measure of Theorem
\ref{12}, is given by
\[ f_{\gamma} (x) = \frac{1}{\mathpi} \tmmathbf{1}_{[\gamma + 1 - 2
   \sqrt{\gamma}, \gamma + 1 + 2 \sqrt{\gamma}]} (x) \arccos \left( \frac{x -
   1 + \gamma}{2 \sqrt{\gamma x}} \right), \]
for $\gamma > 1$ and
\[ f_{\gamma} (x) = \frac{1}{\mathpi} \tmmathbf{1}_{[\gamma + 1 - 2
   \sqrt{\gamma}, \gamma + 1 + 2 \sqrt{\gamma}]} (x) \arccos \left( \frac{x -
   1 + \gamma}{2 \sqrt{\gamma x}} \right) +\tmmathbf{1}_{[0, \gamma + 1 - 2
   \sqrt{\gamma}]} (x), \]
for $\gamma < 1$. This probability distribution is known as the one-sided
Plancherel distribution.

For the cases $q = - 1, 1$, the (shifted) Perelomov-Popov measures that
correspond to certain extreme characters, converge to celebrated probability
distributions such as the semicircle distribution or the Marchenko-Pastur
distribution. For example the $1$-deformed Perelomov-Popov measure of
signatures that are distributed according to the probability measures
described in (\ref{dani}), converges to the probability measure on
$\mathbb{R}$, with $R$-transform $R_{\gamma} (z) = \frac{\gamma}{1 - z}$. This
is the Marchenko-Pastur distribution (or free Poisson distribution) with rate
$\gamma$ (see in {\cite{B15}}). Analogously, for a sequence of random
signatures $(\mathlambda_1 (N) \geq \cdots \geq \mathlambda_N (N)) \in \hat{U}
(N)$ distributed according to the measure $\varrho^{\gamma} (N)$ of
(\ref{dani}), the random measure
\begin{equation}
  m_{N, P \text{} P (- 1)}^{(- \gamma - 1)} [\varrho^{\gamma} (N)] \assign
  \frac{1}{N} \sum_{i = 1}^N \left( \prod_{j \neq i} \frac{(\mathlambda_i (N)
  - i) - (\mathlambda_j (N) - j) + 1}{(\mathlambda_i (N) - i) - (\mathlambda_j
  (N) - j)} \right) \delta \left( \frac{\mathlambda_i (N) - \gamma N - i}{N}
  \right), \label{27}
\end{equation}
converges to the semicircle (or free Gaussian) distribution of variance
$\gamma$,
\[ \tmmathbf{\mu}_{s \text{} c}^{\gamma} (d \text{} x) = \frac{1}{2 \mathpi
   \gamma} \tmmathbf{1}_{[- 2 \sqrt{\gamma}, 2 \sqrt{\gamma}]} (x) \sqrt{4
   \gamma - x^2} d \text{} x. \text{} \]
The Marcheno-Pastur and the semicircle distribution appear in random matrix
theory as limits of the empirical spectral distribution of Wishart and
Gaussian matrices respectively (see in {\cite{B17}}).

Alternatively, from the extreme characters with Voiculescu functions
$\Psi_N^{\gamma} (u) = \exp (N \gamma (u^{- 1} - 1))$, we can get probability
measures $\rho^{\gamma} (N)$ on $\hat{U} (N)$, $N \in \mathbb{N}$, such that
$S_{\rho^{\gamma} (N)} (u_1, \ldots, u_N) = \prod_{i = 1}^N \Psi_N^{\gamma}
(u_i)$. From Theorem \ref{kolpa}, we have that for a sequence of random
signatures $(\mathlambda_1 (N) \geq \cdots \geq \mathlambda_N (N)) \in \hat{U}
(N)$ distributed according to the measure $\rho^{\gamma} (N)$, the random
measure
\begin{equation}
  m_{N, P \text{} P (1)}^{(+ \gamma)} [\rho^{\gamma} (N)] \assign \frac{1}{N}
  \sum_{i = 1}^N \left( \prod_{j \neq i} \frac{(\mathlambda_i (N) - i) -
  (\mathlambda_j (N) - j) - 1}{(\mathlambda_i (N) - i) - (\mathlambda_j (N) -
  j)} \right) \delta \left( \frac{\mathlambda_i (N) + (\gamma + 1) N - i}{N}
  \right), \label{28}
\end{equation}
converges to $\tmmathbf{\mu}_{s \text{} c}^{\gamma}$. The fact that the limit
of (\ref{27}) is equal to the limit of (\ref{28}) should not be a surprise.
More generally, let $X$ be a random variable such that its moments are given
by (\ref{19}) and $Y$ be another random variable such that its moments are
also given by (\ref{19}) if we replace $q$ with $- q$ and the function $\Psi$
with $u \mapsto \Psi (u^{- 1})$. Then we have that $\mathbb{E} [X^k]
=\mathbb{E} [(1 - Y)^k]$ for every $k \in \mathbb{N}$.

In Example \ref{mainexample}, we compute the probability density function of
the limiting measure $m_{N, P \text{} P (q)} [\varrho^{\gamma} (N)]$, for
every $q \in [- 1, 1]$. This gives an interpolation between the probability
distributions mentioned before.

\section{Perelomov-Popov measures and infinitesimal free probability}

In section \ref{.} we showed how from a particular limit regime for the Schur
generating function we can deduce convergence of the $q$-deformed
Perelomov-Popov measures and analyze the limit from the side of free
probability. In this section we consider a different limit regime which
includes the framework of Theorem \ref{12} and Theorem \ref{kolpa} but allows
us to go a step further and compute the correction to $\tmmathbf{\mu}^{(q)}$.
More specifically, in the following we examine the asymptotic additivity not
only for the normalized logarithm of the Schur generating function but also
for the logarithm itself. Through this limit regime we can study the signed
measure $N \left( m_{N, P \text{} P (q)} [\varrho (N)] -\tmmathbf{\mu}^{(q)}
\right)$, as $N \rightarrow \infty$. For $q$=0 the results of this section
have already been presented in {\cite{B41}}.

\subsection{Asymptotic behavior and infinitesimal moments}

In Theorem \ref{12} we justified how from the assumption (\ref{assumption}) it
emerges an asymptotic expansion for the moments of $\mathbb{E} [m_{N, P
\text{} P (q)} [\varrho (N)]]$ and we clarified explicitly the leading term of
this expansion (i.e. the sequence $(M_{0, k, N}^{(q)} (\Psi))_{N \in
\mathbb{N}}$ according to the notation in the proof of Theorem \ref{12}) and
its limit. Now we focus on the second term of this expansion (namely the
sequence $(M_{1, k, N}^{(q)} (\Psi))_{N \in \mathbb{N}}$). Although the limit
regime (\ref{assumption}) gives rise to such a term, it is reasonable to
consider a more general limit regime that leads to a more general formula.
This will have the advantage that it will allow us to compute the moments of
$N (\mathbb{E}[m_{N, P \text{} P (q)} [\varrho (N)]] -\tmmathbf{\mu}^{(q)})$,
as $N \rightarrow \infty$. In order to be done, by our analysis in the proof
of Theorem \ref{12} it is clear that the existence of the limits, $\lim_{N
\rightarrow \infty} \left( \partial_1^l (\log \text{} S_{\varrho (N)}) (1^N) -
N \Psi^{(l)} (1) \right)$, $\lim_{N \rightarrow \infty} \partial_{i_1}^{l_1}
\ldots \partial_{i_s}^{l_s} (\log \text{} S_{\varrho (N)}) (1^N)$ is
essential. The limit regime of the following theorem provides the simplest way
to guarantee it.

\begin{theorem}
  \label{meis}Let $\varrho (N), N \in \mathbb{N}$, be a sequence of
  probability measures on $\hat{U} (N)$. Moreover we assume that there exist
  functions $\Psi, \Phi$ analytic is a neighborhood of $1$ such that for every
  $k \in \mathbb{N}$ fixed,
  \begin{equation}
    \lim_{N \rightarrow \infty} N \left( \frac{1}{N} \log \text{} S_{\varrho
    (N)} (u_1, \ldots, u_k, 1^{N - k}) - \sum_{i = 1}^k \Psi (u_i) \right) =
    \sum_{i = 1}^k \Phi (u_i), \label{29}
  \end{equation}
  where the convergence is uniform in a neighborhood of $1^k$. Then for every
  $k \in \mathbb{N}$ we have
  \[ \lim_{N \rightarrow \infty} N \left( \mathbb{E}_{\mathlambda (N)} \left[
     \int_{\mathbb{R}} t^k m_{N, P \text{} P (q)} [\mathlambda (N)] (d \text{}
     t) \right] - \sum_{m = 0}^k \binom{k}{m} \frac{1}{(m + 1) !}  \left. 
     \frac{\mathd^m}{\mathd u^m} (u^{k - q} (\Psi' (u))^{k - m}) \right|_{u =
     1} \right) \]
  \begin{equation}
    \text{{\hspace{5em}}} = \sum_{m = 0}^{k - 1} \binom{k}{m + 1} \frac{1}{m!}
    \left. \frac{\mathd^m}{\mathd u^m} \left( u^{k - q} \left( \Phi' (u) -
    \frac{1 - q}{2 u} \right) (\Psi' (u))^{k - m - 1} \right) \right|_{u = 1}
    . \label{30}
  \end{equation}
\end{theorem}

\begin{proof}
  Condition (\ref{29}) implies that we have an expansion of the form $N^{k +
  1} M_{0, k, N}^{(q)} (\Psi) + N^k M_{1, k, N}^{(q)} (\Psi) + \cdots + M_{k +
  1, k, N}^{(q)} (\Psi)$ for $\mathcal{D}_k^{U (N), q} S_{\varrho (N)}
  \longdownminus_{u_1 = \cdots = u_N = 1}$. This expansion is acquired in the
  same way as in Theorem \ref{12}. Due to (\ref{29}), the leading term will
  contribute to (\ref{30}) with $\lim_{N \rightarrow \infty} N (M_{0, k,
  N}^{(q)} (\Psi) -\tmmathbf{\mu}_k^{(q)})$, which is equal to
  \[ \sum_{m = 0}^k \binom{k}{m} \frac{1}{(m + 1) !} \lim_{N \rightarrow
     \infty} N \left( \partial_1^m \left( u_1^{k - q} \left( \partial_1 \left(
     \frac{1}{N} \log \text{} S_{\varrho (N)} \right) \right)^{k - m} \right)
     (1^N) - \left. \frac{\mathd^m}{\mathd u^m} (u^{k - q} (\Psi' (u))^{k -
     m}) \right|_{u = 1} \right) \]
  \begin{equation}
    = \sum_{m = 0}^{k - 1} \binom{k}{m + 1} \frac{1}{m!}  \left.
    \frac{\mathd^m}{\mathd u^m} (u^{k - q} \Phi' (u) (\Psi' (u))^{k - m - 1})
    \right|_{u = 1} . \label{31}
  \end{equation}
  In order to prove (\ref{30}) it is necessary to compute $\lim_{N \rightarrow
  \infty} M_{1, k, N}^{(q)} (\Psi)$. The summand
  \[ \frac{k (k - 1)}{2} \left( \prod_{1 \leq i < j \leq N} \frac{1}{u_i -
     u_j} \right)  \sum_{i = 1}^N u_i^{k - 1 - q} \partial_i^{k - 1} \left(
     \prod_{1 \leq i < j \leq N} (u_i - u_j) S_{\varrho (N)} (u_1, \ldots,
     u_N) \right), \]
  of $\mathcal{D}_k^{U (N), q} S_{\varrho (N)} (u_1, \ldots, u_N)$ may not
  play any role for $M_{0, k, N}^{(q)} (\Psi)$ but it does for $M_{1, k,
  N}^{(q)} (\Psi)$. Using the same arguments as in the proof of Theorem
  \ref{12} we have that it contributes to $\lim_{N \rightarrow \infty} M_{1,
  k, N}^{(q)} (\Psi)$ (or equivalently to (\ref{30})) with
  \begin{equation}
    \frac{k - 1}{2} \sum_{m = 0}^{k - 1} \binom{k}{m + 1} \frac{1}{m!}  \left.
    \frac{\mathd^m}{\mathd u^m} (u^{k - 1 - q} (\Psi' (u))^{k - 1 - m})
    \right|_{u = 1} . \label{32}
  \end{equation}
  The remaining part that contributes to $\lim_{N \rightarrow \infty} M_{1, k,
  N}^{(q)} (\Psi)$ comes from the summand of $\mathcal{D}_k^{U (N), q}
  S_{\varrho (N)} (u_1, \ldots, u_N)$ that gives rise to $M_{0, k, N}^{(q)}
  (\Psi)$ (described in the proof of Theorem \ref{12}). On the one hand,
  taking into account that the symmetrized sums (\ref{esta}) for $l_1 = k -
  m$, do not depend on $b_0, \ldots, b_m \in \{1, \ldots, N\}$ when $u_1 =
  \cdots = u_N = 1$, we see that they contribute to $M_{1, k, N}^{(q)} (\Psi)$
  with
  \[ - \sum_{m = 1}^k \frac{k!}{2 (m - 1) !m! (k - m) !} \partial_1^m \left(
     u_1^{k - q} \left( \partial_1 \left( \frac{1}{N} \log \text{} S_{\varrho
     (N)} \right) \right)^{k - m} \right) (1^N) + o (1), \]
  that converges as $N \rightarrow \infty$ to
  \[ \frac{- (k - q)}{2} \sum_{m = 0}^{k - 1} \binom{k}{m + 1} \frac{1}{m!} 
     \left. \frac{\mathd^m}{\mathd u^m} (u^{k - q - 1} (\Psi' (u))^{k - m -
     1}) \right|_{u = 1} \text{ {\hspace{9em}}} \]
  \begin{equation}
    \text{{\hspace{10em}}} - \sum_{m = 0}^{k - 2} \frac{k!}{2 m! (m + 1) ! (k
    - m - 2) !}  \left. \frac{\mathd^m}{\mathd u^m} (u^{k - q} \Psi'' (u)
    (\Psi' (u))^{k - m - 2}) \right|_{u = 1} . \label{33}
  \end{equation}
  On the other hand, the symmetrized sums (\ref{esta}) for $l_1 = k - m - 2$
  and $l_2 = 1$ will also contribute to $M_{1, k, N}^{(q)} (\Psi)$ as $N
  \rightarrow \infty$. Their contribution converges to
  \begin{equation}
    \sum_{m = 0}^{k - 2} \frac{k!}{2 m! (m + 1) ! (k - m - 2) !}  \left.
    \frac{\mathd^m}{\mathd u^m} (u^{k - q} \Psi'' (u) (\Psi' (u))^{k - m - 2})
    \right|_{u = 1} . \label{34}
  \end{equation}
  Note that only for these two cases symmetrized sums of the form
  (\ref{esta}) contribute to $M_{1, k, N}^{(q)} (\Psi)$. Thus (\ref{30}) is
  given by the sum of (\ref{31}), (\ref{32}), (\ref{33}), (\ref{34}), that
  proves the claim.
\end{proof}

\begin{remark}
  \label{vli}An interpolation between the limit regime of Theorem \ref{12}
  and the limit regime of Theorem \ref{meis} can be made if we consider
  the case where for some $0 \leq \varepsilon \leq 1$,
  \begin{equation}
    \lim_{N \rightarrow \infty} N^{\varepsilon} \left( \frac{1}{N} \log
    \text{} S_{\varrho (N)} (u_1, \ldots, u_k, 1^{N - k}) - \sum_{i = 1}^k
    \Psi (u_i) \right) = \sum_{i = 1}^k \Phi (u_i) \label{35}
  \end{equation}
  uniformly in a neighborhood of $1^k$, for every $k \in \mathbb{N}$ fixed.
  Then for $0 < \varepsilon < 1$ and $N \rightarrow \infty$, only the term
  $N^{\varepsilon} (M_{0, k, N}^{(q)} (\Psi) -\tmmathbf{\mu}_k^{(q)})$ will
  contribute to the $k$-th moment of $N^{\varepsilon} \left( \mathbb{E}[m_{N,
  P \text{} P (q)} [\varrho (N)]] -\tmmathbf{\mu}^{(q)} \right)$, which
  implies that it will converge to (\ref{31}).
\end{remark}

\begin{remark}
  Of course the condition (\ref{29}) implies the condition (\ref{assumption}).
  Thus, if (\ref{29}) holds then (\ref{19}) and (\ref{moments}) must coincide.
\end{remark}

\subsection{Asymptotic behavior and infinitesimal $R$-transform}

In Theorem \ref{meis} we showed how additional information for the
asymptotics of the logarithm of the Schur generating function leads to the
determination of the $1 / N$ correction to the limit of Theorem \ref{12}. The
dependence of the limiting probability measure from the function $\Psi$ has
been analyzed and explained from the side of free probability. In this
subsection we describe the free probabilistic toolbox that explains the
dependence of $\Psi$ and $\Phi$ from the $1 / N$ correction of Theorem
\ref{meis}. This fits into the context of infinitesimal free
probability.

Infinitesimal free probability is a generalization of free probability and
typically a theory of non-commutative random variables {\cite{B13}},
{\cite{B16}}, {\cite{B44}}, {\cite{B36}}, {\cite{B14}}, {\cite{B38}},
{\cite{B12}}. For our purposes, it suffices to keep the conversation in the
commutative setting. The key notion for us is that of infinitesimal free
cumulants which leads to an infinitesimal analogue of the $R$-transform.

\begin{definition}[F$\acute{\text{e}}$vrier-Nica, {\cite{B14}}]
  \label{FN}Let $(\tmmathbf{\mu}, \tmmathbf{\mu}')$ be a pair of measures on
  $\mathbb{R}$ where $\tmmathbf{\mu}$ is a probability measure and
  $\tmmathbf{\mu}'$ is a signed measure such that $\tmmathbf{\mu}'
  (\mathbb{R}) = 0$. The infinitesimal free cumulants $\{(\kappa_n
  (\tmmathbf{\mu}))_{N \in \mathbb{N}}, (\kappa_n (\tmmathbf{\mu}'))_{N \in
  \mathbb{N}} \}$ of $(\tmmathbf{\mu}, \tmmathbf{\mu}')$ are uniquely
  determined by the relations
  \[ \int_{\mathbb{R}} t^k \tmmathbf{\mu} (d \text{} t) = \sum_{\pi \in
     \tmop{NC} (k)} \prod_{V \in \pi} \kappa_{|V|} (\tmmathbf{\mu})
     \text{\quad and\quad} \int_{\mathbb{R}} t^k \tmmathbf{\mu}' (d \text{} t)
     = \sum_{\pi \in \tmop{NC} (k)} \sum_{V \in \pi} \kappa_{|V|}
     (\tmmathbf{\mu}') \prod_{\underset{W \neq V}{W \in \pi}} \kappa_{|W|}
     (\tmmathbf{\mu}), \]
  for every $k \in \mathbb{N}$.
\end{definition}

In order to connect the above definition with the framework of Theorem
\ref{meis}, we can think that the measure $\tmmathbf{\mu}$ arises as a
limit of a sequence of probability measures $(\tmmathbf{\mu}_N)_{N \in
\mathbb{N}}$ and $\tmmathbf{\mu}'$ is the $1 / N$ correction to this limit.
This means that
\[ \lim_{N \rightarrow \infty} N \left( \int_{\mathbb{R}} f
   (t)\tmmathbf{\mu}_N (d \text{} t) - \int_{\mathbb{R}} f (t)\tmmathbf{\mu}(d
   \text{} t) \right) = \int_{\mathbb{R}} f (t) \tmmathbf{\mu}' (d \text{} t),
\]
for sufficiently nice functions $f$. If we assume that the above equality
holds for every monomial then this implies that $\lim_{N \rightarrow \infty} N
(\kappa_n (\tmmathbf{\mu}_N) - \kappa_n (\tmmathbf{\mu})) = \kappa_n
(\tmmathbf{\mu}')$ for every $n \in \mathbb{N}$. In such a situation the
measure $\tmmathbf{\mu}'$ acts like a formal derivative of $\tmmathbf{\mu}$
and analogously $(\kappa_n (\tmmathbf{\mu}'))_{n \in \mathbb{N}}$ is a formal
derivative for $(\kappa_n (\tmmathbf{\mu}))_{N \in \mathbb{N}}$. That being
said, the relation in Definition \ref{FN} for the moments of $\tmmathbf{\mu}'$
is quite natural in the sense that it expresses the application of the Leibniz
rule in the moment-cumulant relations.

In what follows, we are making the connection between the $1 / N$ correction
of Theorem \ref{meis} and infinitesimal free probability, determining
explicitly the infinitesimal free cumulants $\{(\kappa_{n, q} (\Psi, \Phi))_{n
\in \mathbb{N}}, (\kappa_{n, q}' (\Psi, \Phi))_{n \in \mathbb{N}} \}$. This
leads to certain conclusions. Of course we have already determined the
sequence $(\kappa_{n, q} (\Psi, \Phi))_{n \in \mathbb{N}}$ and in order to get
an idea for the formula of $(\kappa_{n, q}' (\Psi, \Phi))_{n \in \mathbb{N}}$
one can look on relation (\ref{30}). Then, we observe the same kind of
phenomenon as for $(\kappa_{n, q} (\Psi, \Phi))_{n \in \mathbb{N}}$, namely
for small $n$ we have that $\kappa_{n, q}' (\Psi, \Phi)$ does not depend on
$\Psi$ and it consists of two summands such that the one is linear on $\Phi$ and the other on does not depend on $\Phi$.
For example from (\ref{30}) we have that
\[ \kappa_{1, q}' (\Psi, \Phi) = \Phi' (1) - \frac{1 - q}{2}, \text{\quad}
   \kappa_{2, q}' (\Psi, \Phi) = \Phi'' (1) + \Phi' (1), \]
\[ \text{and\quad} \kappa_{3, q}' (\Psi, \Phi) = \frac{1}{2} \Phi''' (1) +
   \frac{3 + q}{2} \Phi'' (1) \noplus + \frac{1 + q}{2} \Phi' (1) . \]
In order to clarify the exact form of dependence of $\kappa_{n, q}' (\Psi,
\Phi)$ from $\Phi$ we will use the alternative approach for the computation of
the $1 / N$ convergence, which is based on the differential operators
$\mathfrak{D}_k^{U (N), q}$. Before we state our theorem we describe in a few
words what we conclude from it: It is elementary to check that for the empty
signature $\mathlambda (N) = 0^N$ the measure $m_{N, P \text{} P (q)}
[\mathlambda (N)]$ converges to $\beta (1 - q, 1 + q)$ and the $1 / N$
correction to this limit is given by
\[ \lim_{N \rightarrow \infty} N \left( \int_{\mathbb{R}} t^k m_{N, P \text{}
   P (q)} [\mathlambda (N)] (d \text{} t) - \int_{\mathbb{R}} t^k \beta (1 -
   q, 1 + q) (d \text{} t) \right) = \frac{q - 1}{2} k \int_{\mathbb{R}} t^{k
   - 1} \beta (1 - q, 1 + q) (d \text{} t) \]
which immediately implies that the infinitesimal free cumulants are $\left(
\frac{q - 1}{2}, 0, 0, \ldots \right)$. Giving structure to $\mathlambda (N)$,
at random according to probability measures $\varrho (N)$ on $\hat{U} (N)$
that satisfy (\ref{29}), the limiting measure is affected very concretely by
adding the term $\frac{1}{(n - 1) !}  \left. \frac{\mathd^{n - 1}}{\mathd u^{n
- 1}} (e_q (u) \Psi' (e_q (u))) \right|_{u = 0}$ to the $n$-th free cumulant
of $\beta (1 - q, 1 + q)$, for every $n \in \mathbb{N}$ (i.e. like considering
the free convolution with another measure). Analogously the same rule holds
for the $1 / N$ correction, namely it will be affected by adding $\left(
\frac{1}{(n - 1) !}  \left. \frac{\mathd^{n - 1}}{\mathd u^{n - 1}} (e_q (u)
\Phi' (e_q (u))) \right|_{u = 0} \right)_{n \in \mathbb{N}}$ to $\left(
\frac{q - 1}{2}, 0, 0, \ldots \right)$. Consequently, the upshot of assumption
(\ref{29}) is that it reforms the limiting measure and its first order
correction by considering a very particular infinitesimal free convolution.

\begin{theorem}
  \label{MN}Let $\varrho (N)$, $N \in \mathbb{N}$, be a sequence of
  probability measures on $\hat{U} (N)$. Moreover we assume that there exist
  functions $\Psi, \Phi$ analytic in a neighborhood of $1$ such that for every
  $k \in \mathbb{N}$ fixed
  \begin{equation}
    \lim_{N \rightarrow \infty} N \left( \frac{1}{N} \log \text{} S_{\varrho
    (N)} (u_1, \ldots, u_k, 1^{N - k}) - \sum_{i = 1}^k \Psi (u_i) \right) =
    \sum_{i = 1}^k \Phi (u_i), \label{ev}
  \end{equation}
  where the convergence is uniform in a neighborhood of $1^k$. Then, the
  infinitesimal free cumulants of the $1
  / N$ correction $\lim_{N \rightarrow \infty} N \left( \mathbb{E}[m_{N, P
  \text{} P (q)} [\varrho (N)]] -\tmmathbf{\mu}^{(q)} \right)$ are given by
  \begin{equation}
    \kappa_{n, q} (\Psi, \Phi) = \frac{1}{(n - 1) !}  \left. \frac{\mathd^{n -
    1}}{\mathd u^{n - 1}} (e_q (u) \Psi' (e_q (u))) \right|_{u = 0} + \kappa_n
    (\beta (1 - q, 1 + q)) \text{ for every } n \in \mathbb{N}, \label{999}
  \end{equation}
  and
  \begin{equation}
    \kappa_{n, q}' (\Psi, \Phi) = \frac{1}{(n - 1) !}  \left. \frac{\mathd^{n
    - 1}}{\mathd u^{n - 1}} (e_q (u) \Phi' (e_q (u))) \right|_{u = 0} \noplus
    + \kappa_n \left( \delta \left(  \frac{q - 1}{2} \right) \right) \text{
    for every } n \in \mathbb{N}, \label{99}
  \end{equation}
  where $e_q (u) = (1 - q \text{} u)^{- 1 / q}$.
\end{theorem}

\begin{proof}
  By Lemma 5 of {\cite{B41}}, it suffices to show that for every $k \in
  \mathbb{N}$
  \[ \lim_{N \rightarrow \infty} N \left( \mathbb{E}_{\mathlambda (N)} \left[
     \int_{\mathbb{R}} t^k m_{N, P \text{} P (q)} [\mathlambda (N)] (d \text{}
     t) \right] - \right. \text{{\hspace{18em}}} \]
  \[ \text{{\hspace{7em}}} \left. \sum_{m = 0}^k \binom{k}{m} \frac{1}{(m + 1)
     !}  \left. \frac{\mathd^m}{\mathd u^m} \left( e_q (u) \Psi' (e_q (u)) +
     \frac{e_q (u)}{e_q (u) - 1} - \frac{1}{u} \right)^{k - m} \right|_{u = 0}
     \right) \]
  \[ = \sum_{m = 0}^{k - 1} \binom{k}{m + 1} \frac{1}{m!}  \left.
     \frac{\mathd^m}{\mathd u^m} \left( e_q (u) \Phi' (e_q (u)) - \frac{1 -
     q}{2} \right) \left( e_q (u) \Psi' (e_q (u)) + \frac{e_q (u)}{e_q (u) -
     1} - \frac{1}{u} \right)^{k - m - 1} \right|_{u = 0} . \]
  For notation simplicity we write $\Alpha_q (u) = \Psi (e_q (u))$, $\Beta_q
  (u) = \log \left( \frac{e_q (u) - 1}{u} \right)$ and $\Gamma_q (u) = \Phi
  (e_q (u))$. In order to get an expression for the left hand side of the
  above equality we use the same strategy as in the proof of Theorem
  \ref{meis}. As we explained there, applying $\mathfrak{D}_k^{U (N),
  q}$ to $T^{(q)}_{\varrho (N)}$ and setting $u_1 = \cdots = u_N = 0$, we get
  (\ref{21}). Starting from (\ref{21}), in order to compute the desired limit,
  we see that the sum
  \begin{equation}
    - q \frac{k (k - 1)}{2} \sum_{m = 0}^{k - 1} \binom{k - 1}{m} \frac{1}{(m
    + 1) !}  \left. \frac{\mathd^m}{\mathd u^m} \left( e_q (u) \Psi' (e_q (u))
    + \frac{e_q (u)}{e_q (u) - 1} - \frac{1}{u} \right)^{k - 1 - m} \right|_{u
    = 0}, \label{tsig}
  \end{equation}
  and the limit
  \[ \lim_{N \rightarrow \infty} \left( \sum_{m = 0}^k \binom{k}{m}
     \frac{N}{(m + 1) !} \partial_1^m \left( (1 - q \text{} u_1)^{k + 1}
     \left( \partial_1 \left( \frac{1}{N} \log \text{} T_{\varrho (N)}^{(q)}
     \right) \right)^{k - m} \right) (0^N) \right. \text{{\hspace{7em}}} \]
  \begin{equation}
    \text{{\hspace{10em}}} \left. - \sum_{m = 0}^k \binom{k}{m} \frac{N}{(m +
    1) !}  \left. \frac{\mathd^m}{\mathd u^m} \left( (1 - q \text{} u)^{k + 1}
    (\Alpha_q' (u) + \Beta_q' (u))^{k - m} \right) \right|_{u = 0} \right),
    \label{sper}
  \end{equation}
  will contribute. Due to (\ref{ev}) we have that for every $l \in \mathbb{N}$
  \[ \lim_{N \rightarrow \infty} \left( \partial_1^l \log \text{} T_{\varrho
     (N)}^{(q)} (0^N) - N \text{} \Alpha_q^{(l)} (0) - N \text{} \Beta_q^{(l)}
     (0) \right) = \Gamma_q^{(l)} (0) - \Beta_q^{(l)} (0), \]
  which implies that the above limit is equal to
  \begin{equation}
    \sum_{m = 0}^{k - 1} \binom{k}{m + 1} \frac{1}{m!}  \left.
    \frac{\mathd^m}{\mathd u^m} \left( \left( 1 - q \text{} u \right)^{k + 1}
    (\Gamma_q' (u) - \Beta_q' (u)) (\Alpha_q' (u) + \Beta_q' (u))^{k - m - 1}
    \right) \right|_{u = 0} \label{!}
  \end{equation}
  \begin{equation}
    = \sum_{m = 0}^{k - 1} \binom{k}{m + 1} \frac{1}{m!}  \left.
    \frac{\mathd^m}{\mathd u^m} \left( (1 - q \text{} u)^{m + 1} \varphi_q (u)
    (f_q (u) + q)^{k - m - 1} \right) \right|_{u = 0}, \label{spl}
  \end{equation}
  where $\varphi_q (u) = e_q (u) \Phi' (e_q (u)) - e_q (u) (e_q (u) - 1)^{- 1}
  + u^{- 1} - q$ and $f_q (u) = e_q (u) \Psi' (e_q (u)) + e_q (u) (e_q (u) -
  1)^{- 1} - u^{- 1}$. Using the Leibniz rule for the derivative and changing
  the order of summation we see that (\ref{spl}) is equal to
  \[ \sum_{l = 0}^{k - 1} \sum_{n = 0}^{k - l - 1} \sum_{m = l}^{k - n - 1}
     \frac{1}{l!} \binom{k}{l + 1} \binom{k - l - 1}{n} \binom{k - n - 1 -
     l}{m - l} q^n (- q)^{m - l} \left. \frac{\mathd^l}{\mathd u^l} (\varphi_q
     (u) f_q^{k - n - 1 - m} (u)) \right|_{u = 0} \]
  \begin{equation}
    = \sum_{m = 0}^{k - 1} \binom{k}{m + 1} \frac{1}{m!}  \left.
    \frac{\mathd^m}{\mathd u^m} (\varphi_q (u) f^{k - m - 1}_q (u)) \right|_{u
    = 0} . \label{PR}
  \end{equation}
  Note that the approach of the differential operator $\mathfrak{D}_k^{U (N),
  q}$ is in one sense harder for the determination of the $1 / N$ correction
  because in contrast to the function $S_{\varrho (N)}$, for the function
  $T_{\varrho (N)}^{(q)}$ it does not necessarily hold that the derivatives
  $\partial_1^l \partial_2^k \log \text{} T_{\varrho (N)}^{(q)} (0^N)$
  converge to $0$ as $N \rightarrow \infty$. Based on results of {\cite{B41}}
  (see Theorem 15), we know exactly how such derivatives contribute to the
  limit that we try to compute. They do by the limit
  \begin{equation}
    - \lim_{N \rightarrow \infty} \sum_{m = 1}^{k - 1} \binom{k}{m}
    \frac{N}{(m + 1) !} [(\partial_1 + \partial_2)^m - \partial_1^m] \left(
    \left( 1 - q \text{} u_1 \right)^{k + 1} \left( \partial_1 \left(
    \frac{1}{N} \log \text{} T_{\varrho (N)}^{(q)} \right) \right)^{k - m}
    \right) (0^N) . \label{ptn}
  \end{equation}
  In order to compute this limit note that by the definition of $T_{\varrho
  (N)}^{(q)}$ we have that
  \[ \partial_1 \log \text{} T_{\varrho (N)}^{(q)} (u, u, 0^{N - 2}) - N
     \text{} \Beta_q' (u)  - N \text{} \Alpha_q' (u) = \text{ {\hspace{17em}}}
  \]
  \[ \text{{\hspace{9em}}} e_q' (u) \left( \partial_1 \log \text{} S_{\varrho
     (N)} (e_q (u), e_q (u), 1^{N - 2}) - N \Psi' (e_q (u)) \right) +
     \frac{\Beta_q' (0)}{1 - q \text{} u} - 2 \Beta_q' (u), \]
  which implies that
  \[ \lim_{N \rightarrow \infty} \sum_{m = 1}^{k - 1} \binom{k}{m} \frac{N}{(m
     + 1) !} \left( \left. \frac{\mathd^m}{\mathd u^m} \left( 1 - q \text{} u
     \right)^{k + 1} \left( \partial_1 \left( \frac{1}{N} \log \text{}
     T^{(q)}_{\varrho (N)} \right) \right)^{k - m} (u, u, 0^{N - 2})
     \right|_{u = 0} \text{ {\hspace{4em}}} \right. \]
  \[ \text{{\hspace{21em}}} - \left. \left. \frac{\mathd^m}{\mathd u^m} \left(
     \left( 1 - q \text{} u \right)^{k + 1} (\Alpha_q' (u) + \Beta_q' (u))^{k
     - m} \right) \right|_{u = 0} \right) \]
  \[ = \sum_{m = 1}^{k - 1} \binom{k}{m + 1} \frac{1}{m!}  \left.
     \frac{\mathd^m}{\mathd u^m} \left( \left( 1 - q \text{} u \right)^{k + 1}
     \left( \Gamma_q' (u) + \frac{\Beta_q' (0)}{1 - q \text{} u} - 2 \Beta_q'
     (u) \right) (\Alpha_q' (u) + \Beta_q' (u))^{k - m - 1} \right) \right|_{u
     = 0} . \]
  Therefore, taking into account the computation that we have already done for
  (\ref{sper}), we deduce that (\ref{ptn}) is equal to
  \begin{equation}
    \sum_{m = 0}^{k - 1} \binom{k}{m + 1} \frac{1}{m!}  \left.
    \frac{\mathd^m}{\mathd u^m} \left( \left( 1 - q \text{} u \right)^{k + 1}
    \left( \Beta_q' (u) - \frac{\Beta'_q (0)}{1 - q \text{} u} \right)
    (\Alpha_q' (u) + \Beta_q' (u))^{k - m - 1} \right) \right|_{u = 0} .
    \label{Ypo}
  \end{equation}
  A computation similar to the one that we did in order to show that (\ref{!})
  is equal to (\ref{PR}) implies that (\ref{Ypo}) is equal to
  \[ \sum_{m = 0}^{k - 1} \binom{k}{m + 1} \frac{1}{m!} 
     \frac{\mathd^m}{\mathd u^m} \left( \left( \frac{e_q (u)}{e_q (u) - 1} -
     \frac{1}{u} + q - \Beta_q' (0) \right) \right. \text{{\hspace{15em}}} \]
  \begin{equation}
    \text{{\hspace{15em}}} \times \left. \left. \left( e_q (u) \Psi' (e_q (u))
    + \frac{e_q (u)}{e_q (u) - 1} - \frac{1}{u} \right)^{k - m - 1} \right)
    \right|_{u = 0} . \label{Light}
  \end{equation}
  Finally, the last thing is to examine the convergence of the sequence
  $(M_{1, k, N}^{(q)} (\Psi))_{N \in \mathbb{N}}$ of the expansion that we
  mentioned in the proof of Theorem \ref{kolpa}. On the one hand, the limits
  (\ref{fwta}) will play a role for $M_{1, k, N}^{(q)} (\Psi)$ because they
  appear $m! \binom{N}{m + 1}$ times in $\mathfrak{D}_k^{U (N), q} T_{\varrho
  (N)}^{(q)} \longdownminus_{u_1 = \cdots = u_N = 0}$. This fact in
  combination with assumption (\ref{ev}) implies that the contribution of the
  limits (\ref{fwta}) to $\lim_{N \rightarrow \infty} M_{1, k, N}^{(q)}
  (\Psi)$ is
  \begin{equation}
    - \frac{1}{2} \sum_{m = 1}^k \binom{k}{m} \frac{1}{(m - 1) !}  \left.
    \frac{\mathd^m}{\mathd u^m} \left( (1 - q \text{} u)^{k + 1} (\Alpha_q'
    (u) + \Beta_q' (u))^{k - m} \right) \right|_{u = 0} . \label{dnan}
  \end{equation}
  Note that there are also symmetrized sums of another form, that emerge from
  the relation (\ref{22}) and contribute to $M_{1, k, N}^{(q)} (\Psi)$ as $N
  \rightarrow \infty$. These have the form
  \[ \frac{k!N^{k - m - 1}}{2 m! (k - m - 2) !}  \left( \frac{(1 - q \text{}
     u_{b_0})^{k + 1} \partial_{b_0}^2 \left( \frac{1}{N} \log \text{}
     T^{(q)}_{\varrho (N)} \right) \left( \partial_{b_0} \left( \frac{1}{N}
     \log \text{} T^{(q)}_{\varrho (N)} \right) \right)^{k - m - 2}}{(u_{b_0}
     - u_{b_1}) (u_{b_0} - u_{b_2}) \ldots (u_{b_0} - u_{b_m})} \right. +
     \text{{\hspace{10em}}} \]
  \[ \text{{\hspace{12em}}} \ldots \left. + \frac{(1 - q \text{} u_{b_m})^{k +
     1} \partial_{b_m}^2 \left( \frac{1}{N} \log \text{} T^{(q)}_{\varrho (N)}
     \right) \left( \partial_{b_m} \left( \frac{1}{N} \log \text{}
     T^{(q)}_{\varrho (N)} \right) \right)^{k - m - 2}}{(u_{b_m} - u_{b_0})
     (u_{b_m} - u_{b_1}) \ldots (u_{b_m} - u_{b_{m - 1}})} \right) . \]
  The same arguments that we used for the symmetrized sums (\ref{fwta})
  combined with the assumption (\ref{ev}) are sufficient in order to deduce
  that the above symmetrized sums for $u_1 = \cdots = u_N = 0$ will affect
  $\lim_{N \rightarrow \infty} M_{1, k, N}^{(q)} (\Psi)$ by the sum
  \begin{equation}
    \sum_{m = 0}^{k - 2} \frac{k!}{2 m! (m + 1) ! (k - m - 2) !}  \left.
    \frac{\mathd^m}{\mathd u^m} \left( (1 - q \text{} u)^{k + 1} (\Alpha_q''
    (u) + \Beta_q'' (u)) (\Alpha_q' (u) + \Beta_q' (u))^{k - m - 2} \right)
    \right|_{u = 0} . \label{1!.}
  \end{equation}
  Note that $\lim_{N \rightarrow \infty} M_{1, k, N}^{(q)} (\Psi)$ is equal to
  the sum of (\ref{1!.}) and (\ref{dnan}). Taking the sum of (\ref{tsig}),
  (\ref{PR}), (\ref{Light}) and this limit we see that the claim holds.
\end{proof}

\begin{remark}
  In the setting of Theorem \ref{MN}, as we have already mentioned, the $1 /
  N$ correction can be considered as a formal derivative with respect to ``$t
  = 1 / N$'' at $0$. In that sense the differentiation procedure that we
  described in the above theorem affects the part of the limiting measure that
  depends on the function $\Psi$, in terms of the free cumulants, by replacing
  the function $\Psi$ with $\Phi$ for the infinitesimal free cumulants. For
  the remaining part of the limiting measure, i.e. the beta distribution,
  although it acts as a constant its derivative is not zero (meaning $\delta
  (0)$) but $\delta \left( \frac{q - 1}{2} \right)$. This changes for the
  intermediate limit regime that we described in Remark \ref{vli}. In that
  case the correction to the limit plays again the role of a derivative at $0$
  but with respect to ``$t = 1 / N^{\varepsilon}$'' for $0 < \varepsilon < 1$.
  Consequently the $1 / N^{\varepsilon}$ correction will be different because
  the derivative of $\beta (1 - q, 1 + q)$ will be $\delta (0)$, meaning that
  the free cumulants of $\delta \left( \frac{q - 1}{2} \right)$ will not
  contribute to the infinitesimal free cumulants (\ref{99}) anymore while the
  formula for the free cumulants (\ref{999}) will remain the same. Under
  slight modifications the same thing happens for the case $\varepsilon = 1$
  as well. One of the reasons is that the free cumulants of the Dirac measure
  exist in (\ref{99}) due to the definition of the $q$-deformed
  Perelomov-Popov measure and not due to the assumption (\ref{ev}). More
  precisely shifting $m_{N, P \text{} P (q)} [\mathlambda (N)]$ by $\frac{1 -
  q}{2 N}$ to the right, for every $N \in \mathbb{N}$, the limit regime
  (\ref{35}) leads to the same limiting measure with free cumulants
  (\ref{999}), for every $0 \leq \varepsilon \leq 1$, and to the same $1 /
  N^{\varepsilon}$ correction, for every $0 < \varepsilon \leq 1$, for which
  the derivative of $\beta (1 - q, 1 + q)$ is equal to zero.
\end{remark}

\subsection{Infinitesimal distributions and extreme characters}

Concluding this section, we give some concrete examples of (infinitesimal)
distributions that arise as limits, in the framework of Theorem \ref{MN} and
Theorem \ref{meis}, when our Schur generating functions correspond to
particular extreme characters. These examples also appear very naturally in
random matrix theory and in some cases they describe very interesting
phenomena for the spectrum, related to outlier eigenvalues and phase
transitions {\cite{B43}}, {\cite{B41}}, {\cite{B12}}.

Let $\varrho (N)$ be probability measures on $\hat{U} (N)$, $N \in
\mathbb{N}$, such that
\[ S_{\varrho (N)} (u_1, \ldots, u_N) =\tmmathbf{\chi}_N (\tmop{diag} (u_1,
   \ldots, u_N, 1, 1 \ldots)) = \prod_{i = 1}^N \exp (N (u_i - 1) + (u_i -
   1)), \]
for every $N \in \mathbb{N}$ and $u_1, \ldots, u_N \in \mathbb{C}$ on the unit
circle. Then the moments of the limit of $m_{N, P \text{} P (- 1)} [\varrho
(N)]$ and its $1 / N$ correction are given by (\ref{moments}) and (\ref{30})
respectively, for $\Psi (u) = \Phi (u) = u - 1$. This implies that the limit
is a semicircle distribution and the $1 / N$ correction is given by the signed
measure
\[ \tmmathbf{\mu}' (d \text{} x) = \frac{1}{2 \mathpi} \tmmathbf{1}_{(0, 4)}
   (x)  \frac{x^2 - 4 x + 2}{\sqrt{(4 - x) x}} d \text{} x. \]
Note that the ``continuous'' analogue of this example, on the level of random
matrices, is given if we consider the limit of the empirical spectral
distribution and its $1 / N$ correction, for the random matrix $A_N + 2
\text{} I_N + N^{- 1 / 2} B_N$, where $I_N$ is the identity matrix of size $N$
and $A_N, B_N$ are independent GUE matrices of size $N$.

An alternative way to get a quite similar outcome, but for the limit of
$1$-Perelomov-Popov measures is by considering probability measures $\rho (N)$
on $\hat{U} (N)$, $N \in \mathbb{N}$, such that
\[ S_{\rho (N)} (u_1, \ldots, u_N) = \prod_{i = 1}^N \exp (N (u_i^{- 1} - 1) +
   (u_i^{- 1} - 1)), \]
for every $u_1, \ldots, u_N$. Then $m_{N, P \text{} P (1)} [\rho (N)]$
converges to the standard semicircle distribution, shifted by 1 and the $1 /
N$ correction is given by
\[ \tilde{\tmmathbf{\mu}}' (d \text{} x) = \frac{1}{2 \mathpi}
   \tmmathbf{1}_{(- 3, 1)} (x)  \frac{x^2 + x - 2}{\sqrt{(1 - x) (3 + x)}} d
   \text{} x. \]
In the sense that we explained before, this corresponds to the random matrix
$A_N - (1 + N^{- 1}) I_N + N^{- 1 / 2} B_N$.

Infinitesimal limits of the averaged empirical distribution attracted
attention in the context of random matrices because they provide information
for outlier eigenvalues {\cite{B43}}, {\cite{B41}}, {\cite{B44}},
{\cite{B12}}. Similar applications were shown in {\cite{B41}}, in the context
of random domino tiling and asymptotic representation theory, studying the
infinitesimal limit of $m_{N, P \text{} P (0)} [\varrho (N)]$. For certain
sequences of measures on $\hat{U} (N)$, the infinitesimal limits of
$q$-deformed Perelomov-Popov measures, for $q = - 1, 1$, coincide with the
infinitesimal limits of finite rank perturbations of GUE or Wishart matrices.
For example the $1 / N$ correction of $m_{N, P \text{} P (- 1)}
[\varrho^{\gamma, \alpha} (N)]$, where
\[ S_{\varrho^{\gamma, \alpha} (N)} (u_1, \ldots, u_N) = \prod_{i = 1}^N \exp
   (\gamma N (u_i^{- 1} - 1)) \frac{1}{1 - \alpha (u_i - 1)}, \]
is given by a signed measure
\[ \tmmathbf{\mu}_{\gamma, \alpha}' (d \text{} x) =\tmmathbf{1}_{\alpha + 1
   \geq \sqrt{\gamma}}  \frac{\alpha}{\alpha + 1} \delta \left(
   \frac{\gamma}{\alpha + 1} - \gamma + \alpha + 1 \right)
   \text{{\hspace{14em}}} \]
\[ \text{{\hspace{4em}}} + \frac{1}{2 \mathpi} \tmmathbf{1}_{(- 2
   \sqrt{\gamma} - \gamma, 2 \sqrt{\gamma} - \gamma)} (x)  \frac{\alpha (x +
   \gamma) - 2 \alpha (\alpha + 1)}{(\alpha + 1)^2 + \gamma - (\alpha + 1) (x
   + \gamma)} \frac{d \text{} x}{\sqrt{(2 \sqrt{\gamma} - \gamma - x) (x + 2
   \sqrt{\gamma} + \gamma)}} . \]
The above signed measure is strongly related to the $1 / N$ correction of
$1$-rank perturbation of a Gaussian matrix. From the formula for
$\tmmathbf{\mu}_{\gamma, \alpha}'$ we extract information not only for the
outlier, but also for the weight of $N^{- 1} (\mathlambda_1 (N) + N - 1)$.

\begin{example}
  \label{mainexample}We conclude with a general example where we deal with the
  probability measures $\varrho^{\gamma} (N)$ that we defined in subsection
  \ref{Y}, but now we consider the limiting measure of Theorem \ref{kolpa} and
  its $1 / N$ correction for arbitrary $- 1 \leq q \leq 1$. Comparing the two
  (equivalent) formulas that we have for the moments, the formula
  (\ref{moments}) can be more useful in order to clarify the probability
  density function of the limiting measure $\tmmathbf{\mu}_{\gamma}^{(q)}$.
  This is so because using integral representations for the derivatives that
  appear in the formula for the moments, we can end up to the density. For
  this example the $k$-th moment of the limiting measure
  $\tmmathbf{\mu}_{\gamma}^{(q)}$ is equal to
  \[ \sum_{m = 0}^k \frac{k!}{(m + 1) ! (k - m) !}  \frac{\gamma^{k - m}}{2
     \mathpi i} \oint_{|z| = \frac{1}{\sqrt{\gamma}}} \frac{(z + 1)^{k -
     q}}{z^{m + 1}} d \text{} z \]
  \[ = \frac{1}{2 \mathpi i (k + 1)} \oint_{|z| = \frac{1}{\sqrt{\gamma}}}
     \frac{1}{(z + 1)^{q + 1}} \left( \left( \gamma \text{} z + \frac{1}{z} +
     \gamma + 1 \right)^{k + 1} - (\gamma \text{} z + \gamma)^{k + 1} \right)
     d \text{} z. \]
  Doing a typical change of variables in the above integral, we deduce that
  $\tmmathbf{\mu}_{\gamma}^{(q)} (d \text{} t) = f_{\gamma, q} (t) d \text{}
  t$ for $\gamma > 1$ and $- 1 \leq q \leq 1$, where
  \[ f_{\gamma, q} (t) = \frac{\gamma^{q / 2}}{q \text{} \mathpi t^{q / 2}}
     \tmmathbf{1}_{[(\sqrt{\gamma} - 1)^2, (\sqrt{\gamma} + 1)^2]} (t) \sin
     \left( q \arctan \left( \frac{\sqrt{(\gamma + 1 + 2 \sqrt{\gamma} - t) (t
     - \gamma - 1 + 2 \sqrt{\gamma})}}{t + \gamma - 1} \right) \right) . \]
  On the other hand, if $\gamma < 1$ we have that
  $\tmmathbf{\mu}_{\gamma}^{(q)} (d \text{} t) = \widetilde{f_{\gamma, q}} (t)
  d \text{} t$, for $- 1 \leq q < 1$ where
  \[ \widetilde{f_{\gamma, q}} (t) = \left\{\begin{array}{l}
       (\mathpi q)^{- 1} (1 - \sqrt{\gamma})^{- q} \gamma^{q / 2} \sin
       (\mathpi q) \text{, \quad} t \in (0, (1 - \sqrt{\gamma})^2)\\
       \frac{\gamma^{q / 2}}{q \mathpi t^{q / 2}} \sin \left( q \arctan \left(
       \frac{\sqrt{(\gamma + 1 + 2 \sqrt{\gamma} - t) (t - \gamma - 1 + 2
       \sqrt{\gamma})}}{t + \gamma - 1} \right) + q \mathpi \right) \text{,
       \quad} t \in ((1 - \sqrt{\gamma})^2, 1 - \gamma)\\
       \frac{\gamma^{q / 2}}{q \mathpi t^{q / 2}} \sin \left( q \arctan \left(
       \frac{\sqrt{(\gamma + 1 + 2 \sqrt{\gamma} - t) (t - \gamma - 1 + 2
       \sqrt{\gamma})}}{t + \gamma - 1} \right) \right) \text{, \quad} t \in
       (1 - \gamma, (1 + \sqrt{\gamma})^2)\\
       0 \text{, \quad otherwise} .
     \end{array}\right. \]
  Moreover, $\tmmathbf{\mu}_{\gamma}^{(- 1)} (d \text{} t) =
  \widetilde{f_{\gamma, - 1}} (t) d \text{} t + (1 - \gamma) \delta (0)$. Note
  that the formulas of $f_{\gamma, 0}$ and $\widetilde{f_{\gamma, 0}}$
  coincide with those written in subsection \ref{Y}.
  
  Moreover the $1 / N$ correction of this measure can be determined completely
  applying an integration by parts since its $k$-th moment is equal to
  $\frac{q - 1}{2} k \int_{\mathbb{R}} t^{k - 1} \tmmathbf{\mu}_{\gamma}^{(q)}
  (d \text{} t)$.
  
  \begin{figure}[p]
    \centering
    \raisebox{0.0\height}{\includegraphics[width=16.9404433949889cm,height=21.2189999999cm]{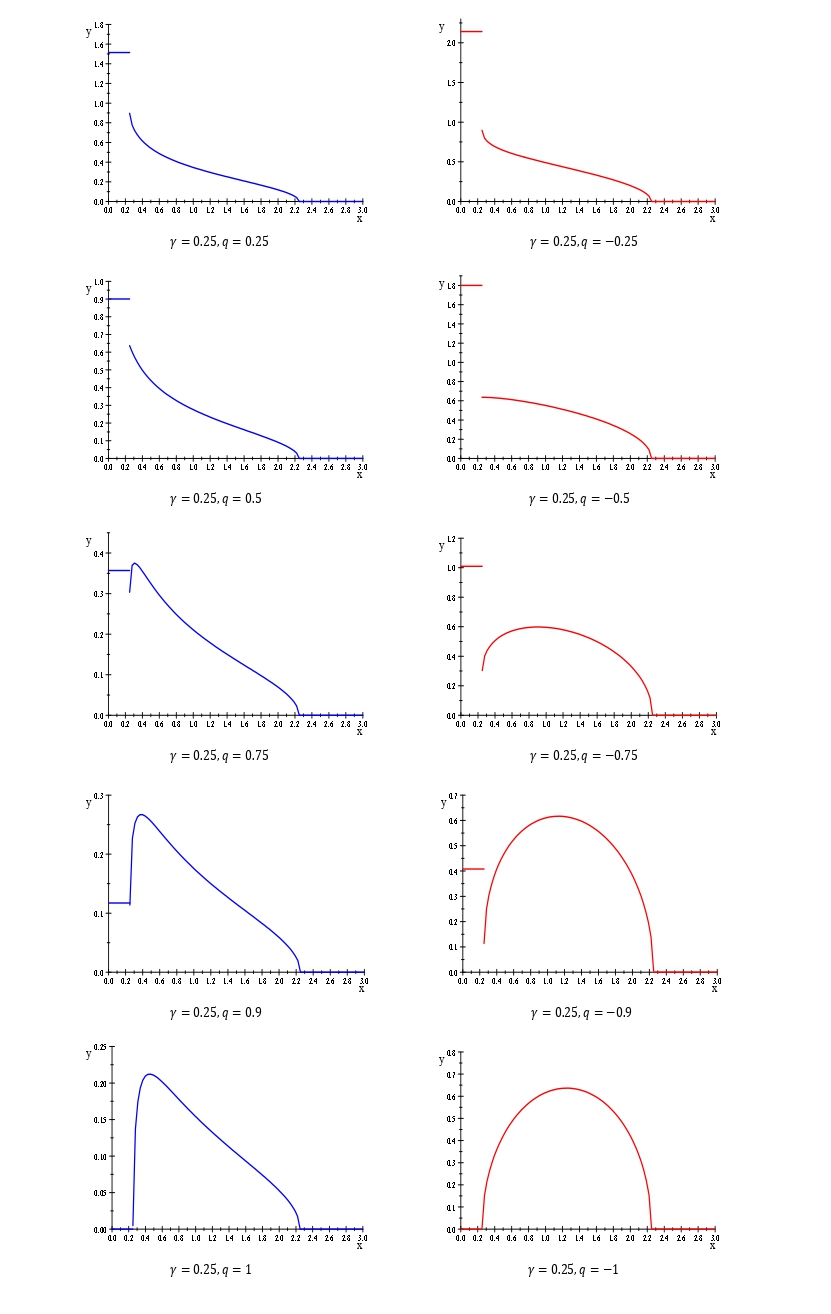}}
    \caption{Graph of $y = \widetilde{f_{\gamma, q}} (x)$ for $\gamma = 0.25$
    and different values of $q \in [- 1, 1]$.}
  \end{figure}
\end{example}

\section{Non-asymptotic relations and Markov-Krein correspondence}

In this section we focus on some non-asymptotic relations regarding the
limiting measures $\tmmathbf{\mu}^{(q)}$ of Theorem \ref{12} and Theorem
\ref{kolpa}. For the time being, we have shown that measures of this form
arise asymptotically as limits of the $q$-deformed Perelomov-Popov measures of
certain random signatures. The connection of Perelomov-Popov measures with
symmetric polynomials provides transforms that allow to express the moments of
$\tmmathbf{\mu}^{(q)}$ in terms of the moments of $\tmmathbf{\mu}^{(q')}$. In
the following we denote by $G_{\tmmathbf{\mu}}$ the Stieltjes transform of a
probability measure $\tmmathbf{\mu}$ on $\mathbb{R}$.

\begin{theorem}
  \label{ata}Let $\varrho (N)$ be a sequence of probability measures on
  $\hat{U} (N)$, $N \in \mathbb{N}$, such that $m_{N, P \text{} P (0)}
  [\varrho (N)]$ converges as $N \rightarrow \infty$ in probability, in the
  sense of moments to a deterministic measure $\tmmathbf{\mu}^{(0)}$ with
  moments $(\tmmathbf{\mu}^{(0)}_k)_{k \in \mathbb{N}}$. Then for every $q \in
  [- 1, 1]$ the probability measures $m_{N, P \text{} P (q)} [\varrho (N)]$
  converge as $N \rightarrow \infty$ in probability, in the sense of moments
  to deterministic probability measures $\tmmathbf{\mu}^{(q)}$ with moments
  $(\tmmathbf{\mu}_k^{(q)})_{k \in \mathbb{N}}$. The moments of
  $\tmmathbf{\mu}^{(q)}$ are related to the moments of $\tmmathbf{\mu}^{(0)}$
  via
  \begin{equation}
    \exp \left( - q \sum_{k = 0}^{\infty} \tmmathbf{\mu}_k^{(0)} z^{k + 1}
    \right) = 1 - q \sum_{k = 0}^{\infty} \tmmathbf{\mu}_k^{(q)} z^{k + 1} .
    \label{MK}
  \end{equation}
\end{theorem}

\begin{proof}
  The claim is a corollary of Proposition \ref{Prop9}.
\end{proof}

\begin{remark}
  In the case that we have probability measures $\varrho (N)$ such that the
  conditions of Theorem \ref{12} are satisfied, assuming that the limiting
  measures have compact supports, the precise formulas for their
  $R$-transforms can also lead to (\ref{MK}) which is basically a relation
  between Stieltjes transforms. To be more precise, the formula for
  $R_{\tmmathbf{\mu}^{(q)}}$,
  \[ R_{\tmmathbf{\mu}^{(q)}} (z) = e_q (z) \Psi' (e_q (z)) + \frac{e_q
     (z)}{e_q (z) - 1} - \frac{1}{z} \text{, \quad with } e_q (z) = (1 - q
     \text{} z)^{- 1 / q}, \]
  implies that
  \[ R_{\tmmathbf{\mu}^{(q)}} \left( \frac{1 - \exp (-
     qG_{\tmmathbf{\mu}^{(0)}} (z))}{q} \right) + \frac{q}{1 - \exp (-
     qG_{\tmmathbf{\mu}^{(0)}} (z))} = R_{\tmmathbf{\mu}^{(0)}}
     (G_{\tmmathbf{\mu}^{(0)}} (z)) + \frac{1}{G_{\tmmathbf{\mu}^{(0)}} (z)} =
     z. \]
  In the previous equality, we used that for an arbitrary probability measure
  $\tmmathbf{\nu}$ on $\mathbb{R}$ with compact support, $z \mapsto
  R_{\tmmathbf{\nu}} (z) + \frac{1}{z}$ is the compositional inverse of
  $G_{\tmmathbf{\nu}}$. Therefore we must have that $1 - \exp (-
  qG_{\tmmathbf{\mu}^{(0)}} (z)) = qG_{\tmmathbf{\mu}^{(q)}} (z)$ or
  equivalently (\ref{MK}).
\end{remark}

\begin{remark}
  The non-asymptotic relation (\ref{infinitesimalSt}), in Theorem
  \ref{infinitesimaltransform} emerges directly from the results of Theorem
  \ref{MN} based on arguments coming from infinitesimal free probability.
  Before we make it clear, note also that for a deterministic sequence
  $\mathlambda (N) \in \hat{U} (N)$ such that $m_{N, P \text{} P (q)}
  [\mathlambda (N)]$ has an infinitesimal limit,
  (\ref{infinitesimaltransform}) can be deduced by Proposition \ref{Prop9}. We
  can deduce the same for the $q$-deformed Perelomov-Popov measures of Theorem
  \ref{MN}, using the standard functional relations between the infinitesimal
  Stieltjes transform, the infinitesimal $R$-transform and the Stieltjes
  transform (of a non-commutative random variable) {\cite{B38}}, combined with
  the formulas that we have proved for the (infinitesimal) free cumulants.
\end{remark}

A natural question to ask is what kind of probability measures are related
through (\ref{MK}). For $q = - 1$, the relation (\ref{MK}) gives the well
known Markov-Krein correspondence {\cite{B35}}, {\cite{B23}}. This is a
bijection between measures on $\mathbb{R}$, with connections to asymptotic
representation theory of the symmetric groups {\cite{B18}} and the Hausdorff
moment problem {\cite{B27}}. For $q = 1$, this is a slight modification of the
Markov-Krein correspondence, firstly introduced in {\cite{B1}} and it can be
derived immediately from the original one. We recall that the Markov-Krein
correspondence states that for every probability measure $\tmmathbf{\mu}$ with
compact support on $\mathbb{R}$, which is absolutely continuous with respect
to the Lebesque measure and its probability density function is bounded by
$1$, there exists a unique probability measure $\tmmathbf{\nu}$, with compact
support on $\mathbb{R}$ such that
\begin{equation}
  \exp \left( \sum_{k = 0}^{\infty} \tmmathbf{\mu}_k z^{k + 1} \right) = 1 +
  \sum_{k = 0}^{\infty} \tmmathbf{\nu}_k z^{k + 1}, \label{ilosp}
\end{equation}
where $(\tmmathbf{\mu}_k)_{k \in \mathbb{N}}, (\tmmathbf{\nu}_k)_{k \in
\mathbb{N}}$ are the moments of $\tmmathbf{\mu}, \tmmathbf{\nu}$ respectively.
It is known (see in {\cite{B23}}) that the map that sends $\tmmathbf{\mu}$ to
$\tmmathbf{\nu}$ is a bijection. We denote it by $M \text{} K_{(- 1)}$.

\begin{corollary}
  Let $\tmmathbf{\mu}$ be a probability measure on $\mathbb{R}$ with compact
  support, which is absolutely continuous with respect to the Lebesque measure
  and its probability density function is bounded by $1 / |q|$. Then, there
  exists a unique probability measure $\tmmathbf{\nu}$ with compact support on
  $\mathbb{R}$ such that
  \[ \exp \left( - q \sum_{k = 0}^{\infty} \tmmathbf{\mu}_k z^{k + 1} \right)
     = 1 - q \sum_{k = 0}^{\infty} \tmmathbf{\nu}_k z^{k + 1}, \]
  where $(\tmmathbf{\mu}_k)_{k \in \mathbb{N}}, (\tmmathbf{\nu}_k)_{k \in
  \mathbb{N}}$ are the moments of $\tmmathbf{\mu}, \tmmathbf{\nu}$
  respectively. Moreover the map $M \text{} K_{(q)}$ that sends
  $\tmmathbf{\mu}$ to $\tmmathbf{\nu}$ is a bijection.
\end{corollary}

\begin{proof}
  The map $M \text{} K_{(q)} \assign \Pi_{- q} \circ M \text{} K_{(- 1)} \circ
  \Pi_{- 1 / q}$ where $\Pi_{\alpha} (\tmmathbf{\mu}) (\cdot) \assign
  \tmmathbf{\mu} (\{ x \in \mathbb{R} \of \alpha x \in \cdot \})$ for $\alpha
  \in \mathbb{R}$, satisfies all the desired properties.
\end{proof}

Combining $M \text{} K_{(q)}$ for $q = \pm 1$ and $- 1 < q < 1$ we can get an
operation of probability measures that translates the classical free
convolution to the $q$-deformed quantized free convolution.

\begin{theorem}
  \label{Xri}Let $q \in (- 1, 0) \cup (0, 1)$. For every probability measure
  $\tmmathbf{\mu}$ with compact support on $\mathbb{R}$ there exist unique
  probability measures $\tmmathbf{\nu} $ and $\tmmathbf{\lambda}$, with
  compact support on $\mathbb{R}$ such that
  \[ 1 - G_{\tmmathbf{\mu}} (z) = (1 - q \text{} G_{\tmmathbf{\nu}} (z))^{1 /
     q} \text{ and } 1 + G_{\tmmathbf{\mu}} (z - 1) = (1 -
     qG_{\tmmathbf{\lambda}} (z))^{- 1 / q} \]
  for $z \in \mathbb{C}$ with $|z|^{- 1}$ small enough. Moreover the map
  $\mathfrak{P}$ that sends such measures $\tmmathbf{\mu}$ to the
  corresponding $\tmmathbf{\nu}$ and the map $\mathfrak{Q}$ that sends
  $\tmmathbf{\mu}$ to the corresponding $\tmmathbf{\lambda}$ are injective and
  such that
  \begin{equation}
    \mathfrak{P} (\tmmathbf{\nu}_1 \boxplus \tmmathbf{\nu}_2) =\mathfrak{P}
    (\tmmathbf{\nu}_1) \otimes_q \mathfrak{P} (\tmmathbf{\nu}_2) \text{\quad
    and\quad} \mathfrak{Q} (\tmmathbf{\nu}_1 \boxplus \tmmathbf{\nu}_2)
    =\mathfrak{Q} (\tmmathbf{\nu}_1) \otimes_q \mathfrak{Q} (\tmmathbf{\nu}_2)
    . \label{qconv}
  \end{equation}
\end{theorem}

\begin{proof}
  We start with the map $\mathfrak{P}$. We define $\mathfrak{P} \assign M
  \text{} K_{(q)} \circ M \text{} K_{(- 1)}^{- 1}$. It is immediate that the
  desired relation between $G_{\tmmathbf{\mu}}$ and $G_{\mathfrak{P}
  (\tmmathbf{\mu})}$ is satisfied and $\mathfrak{P}$ is injective. In order to
  show the property regarding the free convolution, note that for a compactly
  supported probability measure $\tmmathbf{\mu}$ on $\mathbb{R}$, by the
  definition of $\mathfrak{P} (\tmmathbf{\mu})$ we have
  \[ R_{\mathfrak{P} (\tmmathbf{\mu})} (z) = R_{\tmmathbf{\mu}} (1 - (1 - q
     \text{} z)^{1 / q}) + \frac{(1 - q \text{} z)^{- 1 / q}}{(1 - q \text{}
     z)^{- 1 / q} - 1} - \frac{1}{z}, \]
  which implies that $R^{q \text{} u \text{} a \text{} n \text{} t
  (q)}_{\mathfrak{P} (\tmmathbf{\nu}_1 \boxplus \tmmathbf{\nu}_2)} (z) = R^{q
  \text{} u \text{} a \text{} n \text{} t (q)}_{\mathfrak{P} (\tmmathbf{\mu})}
  (z) + R^{q \text{} u \text{} a \text{} n \text{} t (q)}_{\mathfrak{P}
  (\tmmathbf{\nu})} (z)$.
  
  On the other hand, we define $\mathfrak{Q} \assign M \text{} K_{(q)} \circ M
  \text{} K_{(1)}^{- 1} \circ T_1$ where $T_1 (\tmmathbf{\mu}) (\cdot) \assign
  \tmmathbf{\mu} (\{x \in \mathbb{R} \of x + 1 \in \cdot\})$ and by definition
  we have $1 + G_{\tmmathbf{\mu}} (z - 1) = (1 -
  qG_{\mathfrak{Q}(\tmmathbf{\mu})} (z))^{- 1 / q}$. As a consequence
  $R_{\mathfrak{Q} (\tmmathbf{\mu})}$ is given by
  \[ R_{\mathfrak{Q} (\tmmathbf{\mu})} (z) = R_{\tmmathbf{\mu}} ((1 - q
     \text{} z)^{- 1 / q} - 1) + \frac{(1 - q \text{} z)^{- 1 / q}}{(1 - q
     \text{} z)^{- 1 / q} - 1} - \frac{1}{z}, \]
  which implies that $R^{q \text{} u \text{} a \text{} n \text{} t
  (q)}_{\mathfrak{Q} (\tmmathbf{\nu}_1 \boxplus \tmmathbf{\nu}_2)} (z) =
  R_{\mathfrak{Q} (\tmmathbf{\nu}_1)}^{q \text{} u \text{} a \text{} n \text{}
  t (q)} (z) + R_{\mathfrak{Q} (\tmmathbf{\nu}_2)}^{q \text{} u \text{} a
  \text{} n \text{} t (q)} (z)$.
\end{proof}

We denote by $\mathcal{P}_{\leq 1} (\mathbb{R})$ the set of compactly
supported probability measures on $\mathbb{R}$, continuous with respect to the
Lebesque measure and with probability density function bounded by $1$. Note
that the images of $\mathfrak{P}$ and $\mathfrak{Q}$ are the same but these
maps are not surjective. By relation (\ref{qconv}) we see that in order to
define the $q$-deformed quantized free convolution, we have to restrict on
probability measures that belong to the image of $\mathfrak{P}$. In other
words, it is defined on $M \text{} K_{(q)} (\mathcal{P}_{\leq 1}
(\mathbb{R}))$.

It is an open question to determine this set. However, since $\mathfrak{P}$
and $\mathfrak{P}^{- 1}$ are continuous we are able to determine some dense
subsets. First of all, since every probability measure on $\mathcal{P}_{\leq
1} (\mathbb{R})$ is a limit as $N \rightarrow \infty$ of a measure $m_{N, P
\text{} P (0)} [\mathlambda (N)]$, for some $\mathlambda (N) \in \hat{U} (N)$,
Proposition \ref{Prop9} implies that the image of this measure under $M
\text{} K_{(q)}$ is a limit as $N \rightarrow \infty$ of $m_{N, P \text{} P
(q)} [\mathlambda (N)]$. Similarly, starting from another dense subset of
$\mathcal{P}_{\leq 1} (\mathbb{R})$ we get another dense subset of $M \text{}
K_{(q)} (\mathcal{P}_{\leq 1} (\mathbb{R}))$.

\begin{proposition}
  Let $\mathcal{P}_{\alpha, \beta}^{(q)} (\mathbb{R})$ be the set of
  probability measures on $\mathbb{R}$ that are continuous with respect to the
  Lebesque measure and they have probability density functions of the form
  \[ f_{\alpha, \beta}^{(q)} (x) = \left\{\begin{array}{l}
       \frac{\sin (\mathpi q)}{\mathpi q} \prod_{i = 1}^r |x - \alpha_i |^{-
       q} |x - \beta_i |^q \text{, \quad} x \in [\alpha_1, \beta_1] \cup
       \ldots [\alpha_r, \beta_r]\\
       0 \text{, \quad} \tmop{otherwise},
     \end{array}\right. \]
  where
  \[ \alpha_1 < \beta_1 < \alpha_2 < \beta_2 < \cdots < \alpha_r < \beta_r
     \text{\quad and\quad} \sum_{i = 1}^r (\beta_i - \alpha_i) = 1. \]
  The set $\mathcal{P}_{\alpha, \beta}^{(q)} (\mathbb{R})$ is a dense subset
  of $M \text{} K_{(q)} (\mathcal{P}_{\leq 1} (\mathbb{R}))$.
\end{proposition}

\begin{proof}
  Since every measure on $\mathcal{P}_{\leq 1} (\mathbb{R})$ can be
  approximated by uniform measures $\tmmathbf{\mu}_{\alpha, \beta}$ on subsets
  of $\mathbb{R}$ of the form $[\alpha_1, \beta_1] \cup \ldots \cup [\alpha_r,
  \beta_r]$, where $\alpha_1 < \beta_1 < \alpha_2 < \beta_2 < \cdots <
  \alpha_r < \beta_r$ and $\sum_{i = 1}^r (\beta_i - \alpha_i) = 1$, the image
  of these uniform measures under $M \text{} K_{(q)}$ provides a dense subset
  of $M \text{} K_{(q)} (\mathcal{P}_{\leq 1} (\mathbb{R}))$. Because
  \[ G_{\tmmathbf{\mu}_{\alpha, \beta}} (z) = \sum_{k = 0}^{\infty} \sum_{i =
     1}^r \int_{\alpha_i}^{\beta_i} t^k d \text{} t \frac{1}{z^{k + 1}} = \log
     \left( \prod_{i = 1}^r \frac{z - \alpha_i}{z - \beta_i} \right), \]
  we deduce that the Stieltjes transform of $M \text{} K_{(q)}
  (\tmmathbf{\mu}_{\alpha, \beta})$ is
  \[ G_{M \text{} K_{(q)} (\tmmathbf{\mu}_{\alpha, \beta})} (z) = \frac{1}{q}
     - \frac{1}{q} \prod_{i = 1}^r \frac{(z - \beta_i)^q}{(z - \alpha_i)^q} .
  \]
  Therefore, the Stieltjes inversion formula implies that $M \text{} K_{(q)}
  (\tmmathbf{\mu}_{\alpha, \beta}) (d \text{} x) = f_{\alpha, \beta}^{(q)} (x)
  \text{} d \text{} x$.
\end{proof}

\begin{remark}
  The probability measures $\tmmathbf{\nu}$ that belong to the image of
  $\mathfrak{P}$ are basically determined by the property that the function $F
  (z) = 1 - (1 - qG_{\tmmathbf{\nu}} (z))^{1 / q}$, $z \in \mathbb{C}^+$, is a
  Stieltjes transform of a compactly supported probability measure, on
  $\mathbb{R}$, where $\mathbb{C}^+ \assign \{z \in \mathbb{C} \of \text{Im}
  (z) > 0\}$ and without loss of generality $q > 0$. This automatically means
  that $G_{\tmmathbf{\nu}}$ must satisfy more properties than the standard of
  a Stieltjes transform (see in {\cite{B17}}). We recall that one of the
  necessary conditions for $F$ to be a Stieltjes transform is to be analytic,
  which is the case since $G_{\tmmathbf{\nu}}$ is. Moreover $\limsup_{y
  \rightarrow \infty} y |F (i \text{} y) |$ must be equal to one, which is
  true because $\lim_{y \rightarrow \infty} i \text{} yG_{\tmmathbf{\nu}} (i
  \text{} y) = 1$ and $\lim_{z \rightarrow \infty} \left( z - z \left( 1 -
  \frac{q}{z} \right)^{1 / q} \right) = 1$. The last necessary condition is $F
  (\mathbb{C}^+) \subseteq \mathbb{C}^- \assign \{z \in \mathbb{C} \of
  \text{Im} (z) < 0\}$. That's the point where we need more properties for
  $G_{\tmmathbf{\nu}}$ since the fact that $\tmop{Arg} (1 -
  qG_{\tmmathbf{\nu}} (z)) \in (0, \mathpi)$ for every $z \in \mathbb{C}^+$
  does not suffice and instead we should have that $q^{- 1} \tmop{Arg} (1 -
  qG_{\tmmathbf{\nu}} (z)) \in \{\phi + 2 \text{} k \text{} \mathpi \of \phi
  \in (0, \mathpi), k \in \mathbb{Z}\} \cap (0, q^{- 1} \mathpi)$ for every $z
  \in \mathbb{C}^+$.
\end{remark}
\raggedright

\end{document}